\documentclass[11pt,leqno]{article}
\usepackage{ graphicx, amsfonts, amsthm, amsxtra, verbatim, multicol}
\usepackage[mathscr]{euscript}
\makeindex
\begin{document}

\begin{center}
{\LARGE\bf Quasi-modular forms attached to elliptic curves, I
}
\\
\vspace{.25in} {\large {\sc Hossein Movasati}} \\
Instituto de Matem\'atica Pura e Aplicada, IMPA, \\
Estrada Dona Castorina, 110,\\
22460-320, Rio de Janeiro, RJ, Brazil, \\
E-mail:
{\tt hossein@impa.br} \\
{\tt www.impa.br/$\sim$hossein}
\end{center}
\begin{abstract}
In the present text we give a geometric interpretation of quasi-modular forms using moduli 
of elliptic curves with marked elements in their de Rham cohomologies. In this way differential equations
of modular and quasi-modular forms are interpreted as vector fields on such moduli spaces and they can be calculated
from the Gauss-Manin connection of the corresponding universal family of elliptic curves. For the full modular group such a differential equation is calculated and it turns out to be the Ramanujan differential equation between Eisenstein series. 
We also explain  the notion of period map constructed
from elliptic integrals. This turns out to be the bridge between the algebraic notion of a quasi-modular form and the one as a holomorphic
function on the upper half plane. In this way we also get another interpretation, essentially due to Halphen, of the Ramanujan 
differential equation in terms of hypergeometric functions. The interpretation of quasi-modular forms as sections of jet bundles and some related 
enumerative problems are also presented.

\end{abstract}
\newtheorem{theo}{Theorem}
\newtheorem{exam}{Example}
\newtheorem{coro}{Corollary}
\newtheorem{defi}{Definition}
\newtheorem{prob}{Problem}
\newtheorem{lemm}{Lemma}
\newtheorem{prop}{Proposition}
\newtheorem{rem}{Remark}
\newtheorem{conj}{Conjecture}
\newtheorem{exe}{Exercise}
\newcommand\diff[1]{\frac{d #1}{dz}} 
\def\End{{\rm End}}              
\def\hol{{\rm Hol}}
\def\sing{{\rm Sing}}            
\def\spec{{\rm Spec}}            
\def\cha{{\rm char}}             
\def\Gal{{\rm Gal}}              
\def\jacob{{\rm jacob}}          
\def\Z{\mathbb{Z}}                   
\def\O{{\cal O}}                       

\def\C{\mathbb{C}}                   
\def\as{\mathbb{U}}                  
\def\ring{{\sf R}}                             
\def\R{\mathbb{R}}                   
\def\N{\mathbb{N}}                   
\def\A{\mathbb{A}}                   

\def\D{\mathbb{D}}                   
\def\uhp{{\mathbb H}}                
\newcommand\ep[1]{e^{\frac{2\pi i}{#1}}}
\newcommand\HH[2]{H^{#2}(#1)}        
\def\Mat{{\rm Mat}}              
\newcommand{\mat}[4]{
     \begin{pmatrix}
            #1 & #2 \\
            #3 & #4
       \end{pmatrix}
    }                                
\newcommand{\matt}[2]{
     \begin{pmatrix}                 
            #1   \\
            #2
       \end{pmatrix}
    }
\def\ker{{\rm ker}}              
\def\cl{{\rm cl}}                
\def\dR{{\rm dR}}                

\def\hc{{\mathsf H}}                 
\def\Hb{{\cal H}}                    
\def\GL{{\rm GL}}                
\def\pedo{{\cal P}}                  
\def\PP{\tilde{\cal P}}              
\def\cm {{\cal C}}                   
\def\K{{\mathbb K}}                  
\def\k{{\mathsf k}}                  
\def\F{{\cal F}}                     
\def\M{{\cal M}}
\def\RR{{\cal R}}
\newcommand\Hi[1]{\mathbb{P}^{#1}_\infty}
\def\pt{\mathbb{C}[t]}               
\def\W{{\cal W}}                     
\def\Af{{\cal A}}                    
\def\gr{{\rm Gr}}                
\def\Im{{\rm Im}}                
\newcommand\SL[2]{{\rm SL}(#1, #2)}    
\newcommand\PSL[2]{{\rm PSL}(#1, #2)}  
\def\Res{{\rm Res}}              

\def\L{{\cal L}}                     
\def\Aut{{\rm Aut}}              
\def\any{R}                          
\newcommand\ovl[1]{\overline{#1}}    

\def\per{{\sf  pm}}  
\def\T{{\cal T }}                    
\def\tr{{\sf tr}}                 
\newcommand\mf[2]{{M}^{#1}_{#2}}     
\newcommand\bn[2]{\binom{#1}{#2}}    
\def\ja{{\rm j}}                 
\def\Sc{\mathsf{S}}                  
\newcommand\es[1]{g_{#1}}            
\newcommand\V{{\mathsf V}}           
\newcommand\Ss{{\cal O}}             
\def\rank{{\rm rank}}                
\def\diag{{\rm diag}}
\def\BM{{\sf H}}
\def\Fi{{S}}
\def\Ra{{\rm R}}

\def\Q{{\mathbb Q}}                   

\def\P{\mathbb{P}}

\tableofcontents
\section{Introduction}
The objective of this note is to develop the theory of quasi-modular forms in the framework 
of Algebraic Geometry. The request for an algebra which contains the classical modular forms and which is closed
under derivation leads in a natural way to the theory of quasi-modular forms and this can 
be the main reason why the name {\it differential modular form} is more natural 
(see \cite{maro05, ho06-2}). The literature on modular forms and their applications is 
huge and a naive mind may look for similar applications of quasi-modular forms. At the 
beginning of our journey, we may think that we are dealing with  a new theory. 
However, the main examples of quasi-modular forms and their differential equations go back to 19th century,  due to G. Darboux (1878) and G. Halphen (1881). 

The development of the theory of modular forms as holomorphic functions in the 
Poincar\'e upper half plane has shown that many of problems related to modular forms can be proved  if we
take a modular form into Algebraic Geometry and interpret it  in the following way: 
a modular form of weight $m$ is a section of $m$-times tensor power of 
the line bundle $F$ on compactified moduli spaces of elliptic curves, where the fiber of $F$ at the elliptic curve $E$ is defined to be ${\rm Lie}(E)^\vee$. This is equivalent to say that a modular form is 
a  function from the  pairs $(E,\omega)$ to $\k$, where $E$ is an elliptic 
curve defined over a field $\k$ of arbitrary characteristic  and $\omega$ is a 
regular differential form on $E$,  such that 
$f(E,a\omega)=a^{-m}f(E,\omega)$ for all $a\in \k^*$. Some additional properties regarding the 
degeneration of the pair $(E,\omega)$ is also required. The first interpretation can be generalized to
the context of quasi-modular forms using the notion of jet bundles (see Appendix \ref{jetbundle}). 
We found it much more convenient for calculations  to generalize the later interpretation 
to the context of quasi-modular forms. The differential form $\omega$ is replaced with an 
element in the first  algebraic de Rham cohomology of $E$ such that it is not represented by a regular differential form. 
Algebraic de Rham cohomology is introduced for an arbitrary smooth 
variety by A.  Grothendieck  (1966)  in \cite{gro66}  after a work of Atiyah and Hodge (1955). Apart from the multiplicative 
group $\k^*$, we have also the additive group of $\k$ acting on such pairs and the 
corresponding functional equation of a quasi-modular form. In order to use the algebraic  de Rham 
cohomology we have to assume that the field $\k$ is of characteristic zero. 
It turns out 
that the Ramanujan relations between Eisenstein series can be derived from the Gauss-Manin 
connection of  families of elliptic curves and such  series in the $q$-expansion form are uniquely and
recursively determined by the Ramanujan relations.
Looking in this way, we observe that the theory
is not so much new. We find the Darboux-Halphen differential equation which gives rise to the theory of quasi-modular forms for 
$\Gamma(2)$.

We have given some applications of our effort; in order to find differential  and polynomial equations 
for modular forms attached to congruence groups it is sufficient to construct explicit affine coordinates on the moduli of elliptic 
curves enhanced with certain torsion point structure, and then to calculate its Gauss-Manin connection. As the geometrization of modular forms was an important tool in understanding many difficult problems in number theory, such as the Taniyama-Shimura conjecture and its solution which is known under
modularity theorem, I hope that the geometrization presented 
in this text helps to understand many enumerative problems related to quasi-modular, and even modular, forms (see Appendix \ref{8nov2010}). 
However, the main justification in our mind for writing a text which apparently deals with the mathematics of a century ago, is to prepare the ground for a dreaming program: to develop a theory attached to an arbitrary 
family of varieties similar to the modular form theory attached to elliptic curves. 
Apart from  Siegel and Hilbert modular forms attached to Abelian varieties, such a theory is 
under construction for a certain family of Calabi-Yau varieties which appears in mathematical 
physics, see for  instance \cite{ho11}. The text is written in such a way that the way for generalizations
becomes smoother.

The reader who wants to have a fast overview of the text is invited to 
read the small introductions at the beginning of each section. There, we have tried to
write the content of each section in a down-to-earth way. In this way it turns out that the mathematics 
presented in this text go back to a century ago, to mathematicians like  Gauss, Halphen, Ramanujan, Abel, Picard, Poincar\'e and many others. 
The reader gets only a flavor of the history behind the mathematics of  the present text. A full account of the works of all these respected mathematicians would require  a 
deep reading of many treatises that they have left to us. Another relatively fast reading of the present text would be in the following order: 
\S\ref{int01}, \S\ref{int02}, \S\ref{int03}, \S\ref{ramode}, \S\ref{halphensection}, \S\ref{8nov2010}, \S\ref{int04}, \S\ref{int05}, \S\ref{int06} and the entire \S\ref{ellipticsection}.


We assume that the reader has a  basic knowledge in algebraic geometry, complex analysis in 
one variable and Riemann surfaces.
At the end of each sub section the reader finds many exercises with different degrees of difficulty. 
The reader who is interested on the content of this note and who is not worry with the 
details, may skip them. We have collected such exercises in order to avoid statements like 
{\it it is left to the reader}, {\it it is easy to check} and so on. The students are highly 
recommended to do some good piece of the exercises.

The present text arose from my lecture notes at Besse summer school on quasi-modular forms, 2010. 
Here, I would like to thank the organizers, and in particular, Emmanuel Royer and Fran\c cois Martin for the wonderful job they did.
The text is also used in a mini course given at IMPA during the summer school 2011.  
Prof. P. Deligne sent me some comments on the first and final draft of the present text which is essentially  included in 
Appendix \ref{jetbundle}. Here, I would like to 
thank him for his clarifying comments.  My sincere thanks go to Prof. E. Ghys who drew my attention to the works of Halphen. Finally, I would like to apologize from all the people whose works is related to the topic of the present
text and I do not mention them. There is a huge literature on elliptic curves and modular forms and I am sure that I have missed many related works. 

\section{De Rham cohomology of smooth varieties}
\label{derhamsection}
\subsection{Introduction}
\label{int01}
In many calculus books we find tables of integrals and there we never find a formula for elliptic integrals 
\begin{equation}
 \label{odeio}
\int_{a}^b\frac{Q(x)dx}{\sqrt{P(x)}},
 \end{equation}
 where $P(x)$ is a degree three 
 polynomial in one variable $x$ and with real coefficients, for simplicity we assume that it has real roots,
 and $a,b$ are two consecutive roots of $P$ or $\pm\infty$. Since Abel and Gauss it was known that if we choose $P$ randomly (in other words for generic $P$) 
such integrals 
cannot be calculated in terms of until then well-known functions (for a polynomial $P$ with a double root we can calculate
the elliptic integrals, see Exercise \ref{6nov10}. For other particular examples of $P$  we have some formulas calculating elliptic integrals in terms of the values of the Gamma 
function on rational numbers. The Chowla-Selberg theorem, see for instance Gross's article \cite{gro79},  describes this phenomenon in a complete way). 
It was also well known that any such elliptic integral, say it $\int_\delta Q(x)\omega$ with $\omega:=\frac{dx}{\sqrt{P(x)}}$,  
with the integration domain  $\delta=[a,b]$ and polynomial $P$ fixed, is a linear combination  of two integrals 
$\int_\delta \omega$ and $\int_{\delta}x\omega$, that is, it is possible to calculate effectively two numbers $r_1,r_2\in\R$ such that  
\begin{equation}
 \label{boaa}
\int_\delta Q(x)\omega= r_1\int_{\delta}\omega+r_2\int_\delta x\omega.
\end{equation}
For instance, take 
\begin{equation}
\label{8.11.2010}
P(x)=4(x-t_1)^3-t_2(x-t_1)-t_3, \ t_1,t_2,t_3\in\R
\end{equation}
(in \S\ref{weierstrasssection} we explain why we write $P$ in this form). For examples of the equality (\ref{boaa}) see Exercise \ref{maraomid}. 
The equalities there, are written just by neglecting $\int_{\delta}$. 
The next historical step was to consider the integration on the complex domain instead of integration on the real interval 
(this immigration from the real to complex domain is one of the basic tools for Riemann to formulate the so called "Riemann hypothesis"). 
Now, the integration is on any path which
connects to roots or $\infty$ to each other. Soon after it was invented the new variable $y$ representing the quantity $\sqrt{P(x)}$ and 
the integration took place on a closed path in the 
topological space $E:=\{(x,y)\in \C^ 2| y^2=P(x)\}$. The integrand $Q(x)\frac{dx}{y}$  is now called a differential 1-form and the topological space
$E$ is called  an elliptic curve (for a brief description of how the study of integrals contributed to the development of Algebraic Topology, see the introduction of \S\ref{ellipticsection}).
The study of elliptic, and later Abelian and multiple, integrals by Abel, Picard and Poincar\'e prepared the ground for the definition of the de Rham cohomology of varieties and 
its algebraic version due to Atiyah, Hodge and Grothendieck. 
Nowadays, algebraic geometers prefer to avoid the usage of integrals for expressing an equality like (\ref{boaa}). Instead, they say that the algebraic de Rham cohomology of the elliptic curve $y^2=P(x)$ is a two dimensional 
vector space generated by two differential forms $\frac{x^idx}{y},\ i=0,1$.

In this section we want to explain equalities like (\ref{boaa}) without mentioning integrals and domains of integration. Moreover, we work with a ring $\ring$ instead of the field of 
real numbers. We assume a basic knowledge in  Algebraic Geometry. The reader may consult the book of 
Hartshorne \cite{har77}, for any lacking definition or proof. 
We need to work with families of elliptic curves  and so we use
projective and affine schemes over a
commutative ring $\ring$ with multiplicative identity element $1$. We assume  that 
$\ring$ is an integral domain, that is 
$1\not=0$ and it is without zero divisors. 
For simplicity, we take a field of characteristic zero  and we  consider $\ring$ as a 
finitely generated $\k$-algebra.

Let $E$ be a smooth   curve of genus $g$ over $\k$. The algebraic de Rham cohomologies 
$$
H^i_\dR(E), i=0,1,2
$$
are $\k$-vector spaces of dimensions respectively $1,2g$ and $1$. We have $H^0_\dR(E)=\k$, 
an isomorphism
$$
{\rm Tr}: H^2_\dR(E)\cong\k
$$
and a bilinear map
$$
H^1_\dR(E)\times H^1_\dR(E) \to H^2_\dR(E)
$$ 
The map $\rm Tr$ composed with the bilinear map gives us:
$$
\langle\cdot,\cdot\rangle: H^1_\dR(E)\times H^1_\dR(E) \to \k
$$
which is non-degenerate and anti symmetric. We call it the intersection bilinear form.  We have also a natural filtration of $H^1_\dR(E)$ which is called the Hodge filtration:
$$
\{0\}=F^2\subset F^1\subset F^0=H^1_\dR(E)
$$
Its non-trivial piece $F^1$ is generated by a regular differential form (a differential form
 of the first kind). 
The objective of this section is to define all these in a  down-to-earth manner. 
Grothendieck's original article  \cite{gro66} is still a main source for the definition of 
algebraic de Rham cohomology.

\begin{exe}\rm
\label{6nov10}
\begin{enumerate}
 \item 
Calculate the integral (\ref{odeio}) for $P$ with a double root.
\item
For particular examples of $P$ there are some formulas for elliptic integrals in terms of the values of the Gamma function on rational numbers.
Collect some of these formulas and explain why the formula holds. 
\item
Show that $\int_{\delta} \frac{dx}{\sqrt{P(x)}}$, where $\delta$ is an interval between two consecutive real roots of $P$ or $\pm\infty$, and $\deg(P)=4$, 
can be calculated in terms of elliptic integrals. 
\end{enumerate}

\end{exe}

\subsection{Differential forms}
Let $A$ be a $\k$-algebra and  $\ring \to A$ be a morphism of $\k$-algebras. 
Using this morphism, $A$ can be seen as an $\ring$-algebra.  We assume that $A$ as an $\ring$-algebra is finitely generated. 
    
Let $\Omega_{A/\ring}$ denote the module of relative (K\"ahler) differentials, that is, $\Omega_{A/\ring}$ is the quotient of the  $A$-module freely generated by symbols $dr,\ r\in A$, modulo its submodule generated by 
$$
dr,\ r\in \ring,\ d(ab)-adb-bda,\ d(a+b)-da-db,\ a,b\in A.
$$ 
The $A$-module $\Omega_{A/\ring}$ is finitely generated and it is equipped with the derivation 
$$
d:A\to \Omega_{A/\ring},\ r\mapsto dr.
$$
It has the universal property that for any $\ring$-linear 
derivation $D:=A\to M$ into a $A$-module $M$, there is a unique $A$-linear map $\psi: \Omega_{A/\ring}\to M$ such that 
$D=\psi\circ d$.

Let $X=\spec(A)$ and $T=\spec(\ring)$ be the corresponding affine varieties over $\k$ and $X\to T$ be the map obtained by 
$\ring\to A$.  We will mainly use the Algebraic Geometry notation 
$\Omega_{X/T}^1:=\Omega_{A/\ring}$. Let 
$$
\Omega_{X/T}^i={\bigwedge_{k=1}^i}\Omega_{A/\ring}. 
$$ 
be the $i$-th wedge product of $\Omega_{X/T}$ over $A$, that is, $\Omega_{X/T}^i$ is the quotient of the $A$-module 
freely generated by the symbols $\omega_1\wedge\omega_2\wedge\cdots\wedge\omega_i$ modulo its submodule generated  by elements which make $\wedge$ $A$-linear in each $\omega_i$ and 
$$
\omega_1\wedge \cdots\wedge\omega_j\wedge\omega_{j+1}\wedge\cdots\wedge\omega_i=0,\ \  \text{for }\omega_j=\omega_{j+1}
$$ 
It is convenient to define 
$$
\Omega_{X/T}^0:=A
$$
The differential operator
$$
d_i: \Omega_{X/T}^i\to \Omega_{X/T}^{i+1}
$$
is defined by assuming that it is $\ring$-linear and
$$
d_i(a da_1\wedge\cdots\wedge da_i)=da\wedge da_1\wedge\cdots\wedge da_i, \   \ a,a_1,\ldots,a_i\in A.
$$
Sometimes it is convenient to remember that $d_i$'s are defined relative to $\ring$. 
One can verify easily that $d_i$ is in fact well-defined and satisfy all the properties of the classical properties of 
the differential operator on differential forms on manifolds. From now on we drop the subscript $i$ and simply write $d$. If 
$\ring=\k$ is a field then we write $X$ instead of $X/T$. 
\begin{exe}\rm
\begin{enumerate}
 \item 
Prove the universal property of the differential map $d:A\to \Omega_{A/\ring}$. 
\item
Prove the following properties of the wedge product: For $\alpha\in \Omega^{i}_{X/T}, \ \beta\in \Omega^{j}_{X/T}, \ \gamma\in \Omega^{r}_{X/T}$
$$
(\alpha\wedge\beta)\wedge\gamma=\alpha\wedge(\beta\wedge\gamma),
$$
$$
\alpha\wedge\beta\wedge \gamma=(-1)^{ij+jr+ir}\gamma \wedge \beta\wedge \alpha,\ 
$$
\item
Prove that $d\circ d=0$.
\item
For $\alpha\in \Omega^{i}_{X/T}, \ \beta\in \Omega^{j}_{X/T}$ we have:
$$
d(\alpha\wedge \beta)=(d\alpha)\wedge\beta+(-1)^i\alpha\wedge(d\beta). 
$$
\end{enumerate}

\end{exe}

\subsection{De Rham cohomology}
After the definition of differential forms, we get the de Rham complex of $X/T$, namely:
$$
\Omega_{X/T}^0\to \Omega_{X/T}^1\to \cdots \Omega_{X/T}^i\to \Omega_{X/T}^{i+1}\to\cdots
$$
Since $d\circ d=0$, we  can define the de Rham cohomologies
$$
H_{\dR}^i(X/T):=\frac{\ker (\stackrel{d}{\Omega_{X/T}^i\to \Omega_{X/T}^{i+1}})}{\Im (\stackrel{d}{\Omega_{X/T}^{i-1}\to \Omega_{X/T}^{i}}) }.
$$
\begin{exe}\rm
\begin{enumerate}
 \item 
Let $m$ be the number of generators of the $\ring$-algebra $A$. Show that 
for $i\geq m+1$ we have $\Omega_{X/T}^{i}=0$ and hence $H_{\dR}^i(X/T)=0$.
\item
 Let $A=\ring[x_1,x_2,\ldots,x_n]$. In this case, we  usually use the notation $\A_\ring^n:=\spec(A)$. The $A$-module
$\Omega^{1}_{\A^n_\ring}$ is freely generated by the elements $dx_1,dx_2,\ldots,dx_n$. Prove that 
$$
H^i(\A_\ring^n)=0,\ i=1,2,\ldots
$$
This is Exercise 16.15 c, p. 414. of  \cite{ei95}.
\item
Let us come back to the case of an arbitrary $A$. Let $a_1,a_2,\ldots, a_m\in A$ generate the  $\ring$-algebra $A$. Define 
$$
I=\{P\in\ring[x_1,x_2,\ldots,x_m]\mid P(a_1,a_2,\ldots,a_m)=0\}.
$$
The set $I$ is an ideal of $\ring[x_1,x_2,\ldots,x_m]$ and we have 
$$
A\cong \ring[x_1,x_2,\ldots,x_m]/I
$$
$$
\Omega^i_{X/T}\cong \Omega^{i}_{\A^n_\ring}/(dI\wedge \Omega^{i-1}_{\A^n_\ring}+I\Omega^{i}_{\A^n_\ring}  )
$$
where by $dI\wedge \Omega^{i-1}_{\A^n_\ring}+I\Omega^{i}_{\A^n_\ring}$ we mean the $A$-module generated by 
$$
dr_1\wedge \omega_1+r_2\omega_2,\ r_1,r_2\in I,\ \omega_1 \in  \Omega^{i-1}_{\A^n_\ring},\ \omega_2\in \Omega^{i}_{\A^n_\ring} 
$$
\item
Discuss conditions on $A$ such that $H^0(X/T)=\ring$. For instance, show that if $\ring=\k$ is an algebraically closed  
field of characteristic zero and $X$ is an irreducible reduced variety over $\k$ then $H^0(X)=\k$.
\end{enumerate}
 \end{exe}
\subsection{An incomplete elliptic curve, I}
\label{affine}
In this section we find an explicit basis for the de Rham cohomology of the main examples of this text, 
that is, affine elliptic curves in Weierstrass form. The general theory uses the notion of a Brieskorn module 
which is essentially the same as de Rham cohomology. Our main source for this section is \cite{ho06-1, ked08}. 

Let $t_1,t_2,t_3\in\ring$, $P(x)=4(x-t_1)^3-t_2(x-t_1)-t_3\in \ring[x]$ and $f=y^2-P(x)$. Define
$$
A=\ring[x,y]/\langle f\rangle. 
$$
We have
$$
\Omega_{X/T}^1=\Omega^1_{\A^2_\ring}/\langle f \Omega^1_{\A^2_\ring}+\Omega^0_{\A^2_\ring}df \rangle,\ 
\Omega_{X/T}^2=\Omega^2_{\A^2_\ring}/\langle f \Omega^2_{\A^2_\ring}+df\wedge \Omega^1_{\A^2_\ring} \rangle.
$$
We have to say some words about $\Omega_{X/T}^2$; We define the auxiliary $\ring$-module:
$$
V=\Omega^2_{\A^2_\ring}/df\wedge \Omega^1_{\A^2_\ring} \cong \ring[x,y]/\langle f_x,f_y \rangle,
$$
and
$$
\Delta:=27t_3^2-t_2^3.
$$
\begin{prop}
We have 
$$
 \Delta\Omega_{X/T}^2=0
$$

\end{prop}
\begin{proof}
Using the explicit form of $f$, we can easily verify that the $\ring$-module $V$ is freely generated by 
$dx\wedge dy, xdx\wedge dy$ (here we use the fact that $2$ and $3$ are invertible in $\ring$).  Let $M:V\to V,\ A(\omega)=f\omega$. We write $A$ in the basis $dx\wedge dy, xdx\wedge dy$:
$$
M=\begin{pmatrix}
\frac{2}{3}t_1t_2+t_3 & 
-\frac{2}{3}t_1^2t_2+\frac{1}{18}t_2^2 \\
\frac{2}{3}t_2 &
-\frac{2}{3}t_1t_2+t_3
\end{pmatrix}.
$$
Let $p(z):=z^2-{\mathrm tr}(M)z+\det(M)=\det(M-zI_{2\times 2})$ be the characteristic polynomial of $M$. 
We have $P(f)V=0$ and 
since $\Omega_{X/T}^2=V/\langle f\rangle$, we conclude that $\det(M)\Omega_{X/T}^2=0$.
From another side $\det(A)=\frac{-1}{27}\Delta$.
\end{proof}
By the above proposition $\Delta\omega, \omega\in \Omega^1_{X/T}$  is closed.  
From now on we assume that $\Delta$ is irreducible in $\ring$ and we replace $\ring$ with its localization
on its multiplicative group generated by $\Delta$. Therefore, $\Delta$ is invertible 
in $\ring$ and we can talk about the  pole or zero order along $\Delta$ of an element in any $\ring$-module.
In this way 
$\Omega^2_{X/T}=0$ and 
$$
H_{\dR}^1(X/T)\cong\Omega^1_{\A^2_\ring}/\langle f \Omega^1_{\A^2_\ring}+df\Omega^0_{\A^2_\ring}+ d\Omega^0_{\A^2_\ring}\rangle. 
$$
There are two polynomials  $A,B\in \ring[x]$ such that $AP+BP'=\Delta$. We define
$$
\omega=\frac{1}{\Delta}(Aydx+2Bdy)
$$
which satisfies:
\begin{equation}
 \label{dxdy}
dx=y\omega, dy=\frac{1}{2}P'\omega
\end{equation}
We denote by $\frac{dx}{y}$ and $\frac{xdx}{y}$ the elements $\omega$, respectively $x\omega$. Note that these two elements have poles of order at most one along $\Delta$. 
\begin{prop}
\label{redbull}
The $\ring$-module $H_{\dR}^1(X/T)$ is freely generated by the elements $\frac{dx}{y}$ and $\frac{xdx}{y}$.
\end{prop}
\begin{proof}
Using the equalities (\ref{dxdy}) and $y^2=P(x)$, every element of  $H_{\dR}^1(X/T)$ can be written in the form $(C+yD)\omega,\ C,D\in\ring[x]$. Since $Dy\omega=Ddx$ is exact, this reduces to $C\omega$.
From another side the elements
$$
d(x^ay)=(\frac{1}{2}P'x^a+ax^{a-1}P)\omega
$$
are cohomologus to zero. If $\deg(C)\geq 2$, we can choose a monomial $F=x^a$ in such a way that the leading coefficient of $(\frac{1}{2}P'F+F'P)$ is equal to the leading coefficient of $C$. We subtract $d(Fy)$ from $C\omega$ and we get smaller degree for $C$. We repeat this until getting a degree one $C$.  
\end{proof}
\begin{exe}\rm
\label{maraomid}
Verify the following equalities in $H_{\dR}^1(X/T)$:
$$
\frac{x^2dx}{y}=(2t_1)\frac{xdx}{y}+(-t_1^2+\frac{1}{12}t_2)\frac{dx}{y}
$$
$$
 \frac{x^3dx}{y}=(3t_1^2+\frac{3}{20}t_2)\frac{xdx}{y}
  +(-2t_1^3+\frac{1}{10}t_1t_2+\frac{1}{10}t_3)\frac{dx}{y}
$$
$$
\frac{x^4dx}{y}=(4t_1^3+\frac{3}{5}t_1t_2+\frac{1}{7}t_3) \frac{xdx}{y}+ (-3t_1^4-\frac{1}{10}t_1^2t_2+\frac{9}{35}t_1t_3+\frac{5}{336}t_2^2)\frac{dx}{y}
$$
$$
 \frac{x^5dx}{y}=(5t_1^4+\frac{3}{2}t_1^2t_2+\frac{5}{7}t_1t_3+\frac{7}{240}t_2^2)\frac{xdx}{y}+
(-4t_1^5-\frac{2}{3}t_1^3t_2+\frac{2}{7}t_1^2t_3+\frac{19}{420}t_1t_2^2+\frac{1}{30}t_2t_3)\frac{dx}{y}
$$
\end{exe}

%
\subsection{An incomplete elliptic curve, II}
\label{intaffine}
Let  $P(x)\in\ring[x]$ be as in the previous section and
$$
A=\ring[x,y,z]/\langle y^2-P(x), yz-1\rangle 
$$
We will simply write $\frac{1}{y}$ instead of $z$.  
\begin{prop}
\label{8may2010}
The $\ring$ module $H_{\dR}^1(X/T)$ is freely generated by
$$
\frac{dx}{y}, \frac{xdx}{y}, \frac{dx}{y^2}, \frac{xdx}{y^2}, \frac{x^2dx}{y^2}, 
$$
\end{prop}
\begin{proof}
In this example $dy=\frac{1}{2y}P'dx$ and so every element  $\omega$ of $H_{\dR}^1(X/T)$ can be written in the form 
$Cy^{-k}dx+ Cy^{-k+1},\ C,D\in \ring[x],\ k\geq 1$. We use the equality
$$
d(x^ay^{-b})=ax^{a-1}y^{-b}dx+\frac{-b}{2}x^ay^{-b-2}P'dx
$$
for $b=-1,-2,\ldots$ and see that $\omega$ is reduced to a form with $k=1$ 
(each time we multiply $\omega$ with $\Delta= BP'+Ay^2$). Now, for terms $Cy^{-2}dx$  we make the division of $C$  by $P$ and 
we are thus left with the generators $\frac{dx}{y^2}, \frac{xdx}{y^2}, \frac{x^2dx}{y^2}$. For terms $Dy^{-1}dx$ we proceed as 
in Proposition \ref{redbull} and we are left with the generators $\frac{dx}{y}, \frac{xdx}{y}$.
\end{proof}
 
\subsection{De Rham cohomology of projective varieties}
\label{16may2010}
Let $X$ be a projective reduced variety over $\ring$. 
We have the complex of sheaves of differential forms $(\Omega^\bullet_X,d)$ and we define the $i$-th de Rham cohomology of $X$ as the $i$-th hypercohomology of the complex $(\Omega^\bullet_{X/T},d)$, that is
$$
H_{\dR}^i(X)=\uhp^i(\Omega^\bullet_{X},d).
$$
It is far from the objectives of this text to define the hypercohomology. Instead, we explain how its elements look like and how to calculate it.

Let ${\cal U}=\{ U_i\}_{i\in I}$ be any  open covering of $X$ by affine subsets, where $I$ is a totally ordered finite set. 
We have the following double  complex
\begin{equation}
\label{geble}
\begin{array}{ccccccc}
\vdots   & & \vdots  & & \vdots  \\
\uparrow & &\uparrow & & \uparrow\\ 
\Omega^0_2&\rightarrow
 &\Omega^1_2&\rightarrow & \Omega^2_2& \rightarrow & \cdots \\
\uparrow & &\uparrow & & \uparrow & & \\
\Omega^0_1&\rightarrow
 &\Omega^1_1&\rightarrow & \Omega^2_1& \rightarrow & \cdots \\
\uparrow & &\uparrow & & \uparrow & &\\
\Omega^0_0&\rightarrow
 &\Omega^1_0&\rightarrow & \Omega^2_0& \rightarrow &\cdots
\end{array}
\end{equation}
Here $\Omega^i_j$ is the product over $I_1\subset I,\
\#I_1=j+1$ of the set of global sections $\omega_{\sigma}$
of $\Omega^i_{X/T}$ in the open set $\sigma=\cap_{i\in I_1} U_i$. 
The horizontal arrows are usual differential operator $d$ of
$\Omega^i_{X}$'s and vertical arrows are differential operators $\delta$
in the sense of Cech cohomology, that is,
\begin{equation}
\label{delta}
\delta:\Omega_{j}^i\to \Omega^{i}_{j+1},\ \ 
\{\omega_\sigma\}_{\sigma}\mapsto \{\sum_{k=0}^{j+1}(-1)^k\omega_{\tilde \sigma_k}\mid_{\tilde \sigma}\}_{\tilde \sigma}.
\end{equation}
Here $\tilde \sigma_k$ is obtained from $\tilde\sigma$, neglecting the $k$-th open set in the definition of $\tilde \sigma$. 
The $k$-th piece of the total chain
of (\ref{geble}) is
$$
{\cal L}^k:=\oplus_{i=0}^{k}\Omega^i_{k-i}
$$
with the  differential operator 
\begin{equation}
\label{piorenemigo}
d'=d+(-1)^k\delta:{\cal L}^k\rightarrow
{\cal L}^{k+1}.
\end{equation}
The hypercohomology $\uhp^k(M,\Omega^\bullet)$ is  the total cohomology of
the double complex (\ref{geble}), that is
$$
\uhp^k(M,\Omega^\bullet)=\frac{ \ker( {\cal L}^k\stackrel{d}{\rightarrow}
{\cal L}^{k+1}) }{ \Im( {\cal L}^{k-1}\stackrel{d}{\rightarrow}
{\cal L}^{k}) }.
$$
\begin{exe}\rm
\begin{enumerate}
 \item 
Show that the above definition does not depend on the choice of covering, that is, 
if ${\cal U}_1$ and ${\cal U}_2$ are two open covering of $X$ then the corresponding hypercohomologies are isomorphic in a canonical 
way.
\item 
For which varieties $X$, we have $H^0_\dR(X)=\ring$.  
\end{enumerate}
\end{exe}

 \subsection{Complete elliptic curve, I}
\label{6may10}
Let us consider the projective variety 
$$
E={\rm Proj}(\ring[x,y,z]/\langle zy^2-4(x-t_1z)^3+t_2(x-t_1z)z^2+t_3z^3\rangle )
$$
which is covered by two open sets 
$$
U_0=\spec( \ring[x,y]/\langle y^2-4(x-t_1)^3+t_2(x-t_1)+t_3\rangle), 
$$
$$
U_1= \spec (\ring[x,z]/\langle z-4(x-t_1z)^3+t_2(x-t_1z)z^2+t_3z^3\rangle).
$$
Note that $U_0$  and $U_0\cap U_1$ are the affine varieties in \S \ref{affine}, respectively \S\ref{intaffine}. For simplicity, 
we will drop $/\ring$ and $/T$ from our notations. The variety $E$ has a closed point 
$O:=[0;1;0]$ which is in the 
affine chart
$U_1$. It is sometimes called the point at infinity. 
By definition, we have 
$$
H^1_\dR(E)=\{(\omega_0,\omega_1)\in \Omega^1_{U_0}\times  \Omega^1_{U_1}\mid \omega_1-\omega_0\in d(\Omega^0_{U_0\cap U_1})\}/
d\Omega^0_{U_0}\times d\Omega^0_{U_1}
$$
$$
H^2_\dR(E)=\Omega^2_{U_0\cap U_1}/(\Omega^2_{U_0}+\Omega^2_{U_1}+d\Omega^1_{U_0\cap U_1})
$$
In the definition of $H^1_\dR(E)$ as above, we sometimes take $U_1$ smaller but always 
containing $O$. 

\subsection{Residue calculus}
We need to carry out some residue calculus near the closed smooth point $O$, see for 
instance \cite{tat68}. Such a machinery is usually developed for curves over a field and so 
it seems to be necessary to consider the elliptic curve $E$ over the fractional field $\k_1$ of $\ring$, that is, we  use
$E\otimes_\ring \k_1$ instead of $E$. However, most of our calculations lead to elements in $\ring$ which will be used later in 
the theory of quasi-modular forms.

A regular function $t$ in a neighborhood of $O=[0;1;0]$ is called a coordinate system at $O$ if 
$t(O)=0$ and $t$  generates the one dimensional $\k_1$-space $m_O/m_{O}^2$,  where $m_O$ 
is the ring of regular functions  in a neighborhood of $O$ such that they vanish at $O$, and 
$m_{O}^2$ is the $\O_{X,O}$-module generated by  $ab,\ a,b\in m_O$. Recall that $O$ 
is a smooth  point of $E$. Any meromorphic function $f$ (meromorphic 1-form $\omega$) near $O$ has an 
expansion in $t$:
\begin{equation}
 \label{xeremhava}
f=\sum_{i=-a}^\infty f_i t^i,\ \hbox{resp.} \ \ \omega=(\sum_{i=-a}^\infty f_i t^i )dt,\  
f_i\in\k_1, 
\end{equation}
where $a$ is some integer. The stalk of the ring of meromorphic 
differential $1$-forms at $O$ is a  $\O_{X,O}$-module generated by $dt$ and so $\omega =f dt$ 
for some meromorphic function near 
$O$. Therefore, it is enough to explain the first equality.  Let $a$ be the pole order of $f$ at $O$. We work 
with $t^af$ and so without loosing the generality we can assume that $f$ regular at $O$. Let $f_0=f(O)$.  
For some $f_1\in\k_1$ the difference $=f-f_0-f_{1}t\in m_O^ 2$. We repeat this process get a sequence $f_0,f_1,f_2,\ldots,f_m,f_{m+1},\ldots \in \k_1$ such that $$
f-\sum_{i=0}^mf_it^ i\in m_O^{m+1}
$$
Another way of reformulating the above statement is: 
$$
f=\sum_{i=0}^mf_it^ i+O(t^{m+1}),
$$
where $O(t^i)$ means a sum $\sum_{j\geq i} a_j t^j$. This is what what we have written in  (\ref{xeremhava}). 

 The residue of $\omega$ at $O$ is defined to be $f_{-1}$. 
It is independent of the choice of the 
coordinate $t$. In our example, we usually take the coordinate $t=\frac{x}{y}$ with the 
notation of chart $U_0$ (in the chart 
$U_1$  we have $t=x$). The expansions of $x$ and $y$ in $t$ are of the form: 
\begin{equation}
 \label{residue}
x=\frac{1}{4}t^{-2}+O(t^0),\ y=\frac{1}{4}t^{-3}+O(t^{-1}).
\end{equation}
   \begin{exe}\rm
 \begin{enumerate}
\item
Show that $O$ is a smooth point of $E$, that is, the $\k_1$-vector space $m_{O}/m_{O}^2$ is 
one dimensional.  
\item 
Verify the details of the residue calculus presented in this section. In particular, justify the equalities (\ref{residue}) and prove that the notion of residue does not depend on the coordinate system $t$.
\item
Calculate the residue of $\frac{x^ndx}{y^2},\ n=0,1,2,3,4,5$  at $O$.  
\item
Calculate the first $4$ coefficients of the expansion of $\frac{dx}{y}$ in the coordinate $t=\frac{x}{y}$. 
\item
Let us take the coordinates $(x,z)$ in which the elliptic curve $E$ is given by $z-4(x-t_1z)^3-t_2(x-t_1z)z^2-t_3z^3$ and we have $O=(0,0),\ t=x$. Consider $E$ over the ring $\ring$. A regular function $f$ at $O$ can be written as $\frac{P(x,z)}{Q(x,z)}$ with $Q(0,0)\not =0$. Show that if $Q(0,0)$ is invertible in $\ring$ and $P,Q\in\ring[x,z]$ then all the coefficients in the expansion of $f$ belong to $\ring$ (Hint: Verify this for $f=z$.)
\end{enumerate}
\end{exe}

\subsection{Complete elliptic curve, II}
In this section we prove the following proposition:
\begin{prop}
\label{30may10}
 The canonical restriction map 
$$
H^1_{dR}(E)\to H^1_{dR}(U_0),\ \omega_0\oplus \omega_1\to \omega_0
$$
is an isomorphism of $\ring$-modules. 
\end{prop}
\begin{proof}
 First we check that it is injective. Let us take an element 
$(\omega_0,\omega_1)\in H_\dR^1(X)$  with $\omega_0=0$. By definition $\omega_1=\omega_1-\omega_0=df,\ 
f\in \Omega^0_{U_0\cap U_1}$. Since $\omega_1$ has not poles  on the closed point $O\in X$, $f$ has not too, 
which implies that $(\omega_0,\omega_1)$ is cohomologous to zero.

Now, we prove the surjectivity. The restriction map is $\ring$-linear and so by 
Proposition \ref{redbull}, it is enough to prove that 
$\frac{dx}{y},\ \frac{xdx}{y}$ are in the image of the restriction map. In fact, the corresponding elements 
in $H_\dR^1(E)$ are respectively
$$
(\frac{dx}{y}, \frac{dx}{y}) , \ \ (\frac{xdx}{y}, \frac{xdx}{y}-\frac{1}{2}d(\frac{y}{x})). 
$$
We prove this affirmation for $\frac{xdx}{y}$. We define $\tilde U_1=U_1\backslash \{x=0\} $ and use the definition of hypercohomology 
with  the covering $\{U_0,\tilde U_1\}$.   We compute $x$ and $y$ in terms of the local coordinate 
$t=\frac{y}{x}$ around the point at infinity $O$ and we  have (\ref{residue}). 
 Substituting this in $\frac{xdx}{y}$, we get the desired result. 
\end{proof}
Let $U_0,U_1$ be an arbitrary covering of $E$. We have a well-defined  map
$$
{\rm Tr}: H^2_\dR(X)\to \ring,\ 
\rm Tr(\omega)=\text{ sum of the residues of $\omega_{01}$ around the points $X\backslash U_0$},  
$$
where $\omega$ is represented by  $\omega_{01}\in\Omega^1_{U_0\cap U_1}$. As usual, we take 
the canonical charts of $X$ described in \S \ref{6may10}. 
The map $\rm Tr$ turns out to be an isomorphism of $\ring$-modules. 
\begin{prop}
The $\ring$-module $H^2_\dR(X)$ is of rank one.
\end{prop}
\begin{proof}
 According to Proposition \ref{8may2010} any element in $\Omega^1_{U_0\cap U_1}$ modulo exact forms can be reduced to 
an $\ring$-linear combination of $5$ elements. All these elements are zero in $H^2_\dR(X)$, except the last one 
$\frac{x^2dx}{y^2}$. The first two elements are regular forms in $U_0$ and the next two forms are regular in $U_1$.
We have proved that   any element $\omega\in H^2_\dR(X)$ is reduced to $r\in \frac{x^2dx}{y^2}, \ r\in\ring$. 
Since $\frac{x^2dx}{y^2}$ at $O$ has the residue $\frac{-1}{2}$ (use the local coordinate $t=\frac{x}{y}$ and the 
equalities (\ref{residue})), we get the desired result.
\end{proof}
\begin{exe}\rm
\begin{enumerate}
 \item 
 Let us take two open sets $U_1,\tilde U_1\subset E$ which contain $O$. Show that the definition of de Rham cohomologies of 
$E$ attached to the coverings $\{U_0,U_1\}$ and $\{U_0,\tilde U_1\}$ are canonically isomorphic. 
\item
By our definition of residue, it takes values in $\k_1$, the fractional field of $E$. Show that the map
$\rm Tr$ has values in $\ring$. 
\end{enumerate}
\end{exe}

\subsection{The intersection form}
\label{intersectionform}
Let $X$ be a smooth irreducible reduced projective variety over $\ring$. 
One can define the cup product
\begin{equation}
 \label{cup}
H^i_\dR(X)\times H^j_\dR(X) \to H^{i+j}_\dR(X)
\end{equation}
which is the translation of the usual wedge product for the de Rham cohomologies of real manifolds.  Further, we can define an isomorphism  
\begin{equation}
\label{trace}
{\rm Tr}: H^{2n}_\dR(X)\cong\ring
\end{equation}
of $\ring$-modules which imply that $H^{2n}_\dR(X)$ is free of rank one. For $i=j=\dim(X)$, the map (\ref{cup}) composed with 
(\ref{trace}) gives us a bilinear maps
$$
\langle \cdot,\cdot\rangle: H^n_\dR(X)\times H^n_\dR(X) \to \ring.
$$
We have already defined {\rm Tr} in the case of elliptic curves.  
In this section we are going to define the cup product in  the case of a curve.

Let us take two elements $\omega,\alpha\in H_\dR^{1}(X)$. We take an arbitrary covering $X=\cup_{i}U_i$ of $X$ and we assume 
that $\omega$ and $\alpha$ are given by 
$\{\omega_{ij}\}_{i,j\in I}$
and 
$\{\alpha_{ij}\}_{i,j\in I}$
with
$$
\omega_{j}-\omega_i=df_{ij},\ \alpha_j-\alpha_i=dg_{ij}.
$$
We define 
$$
\gamma:=\omega\cup \alpha\in H^{2}_\dR(X)
$$
which is given by: 
\begin{equation}
 \label{19may2010}
\gamma_{ij}=g_{ij}\omega_j-f_{ij}\alpha_j+f_{ij}dg_{ij}
\end{equation}
Let us consider  the situation of \S \ref{6may10}. In this case 
$$
\frac{dx}{y}\cup \frac{xdx}{y}=\{\omega_{01}\},\ \omega_{01}=\frac{-1}{2}\frac{dx}{x},
$$
and
\begin{equation}
\label{08.05.2010}
\langle \frac{dx}{y}, \frac{xdx}{y}\rangle = 1.
\end{equation}
\begin{exe}\rm
\begin{enumerate}
 \item 
Show that the definition of $\omega\cup \alpha$ does not depend on the covering of the curve $X$ and that 
$\cup$ is non-degenerate. 
\item
For a curve over complex numbers show its algebraic de Rham cohomology, cup product and intersection form are essentially the same objects defined by $C^\infty$-functions.
 
\end{enumerate}

\end{exe}

\section{Gauss-Manin connection}
\subsection{Introduction}
\label{int02}
In 1958 Yu. Manin solved the Mordell conjecture over function fields and A. Grothendieck after reading his article invented
the name Gauss-Manin connection. I did not find any simple  exposition of  this subject, the one which could be understandable
by Gauss's mathematics. I hope that the following explains the presence of the name of Gauss on this notion. Our story again goes back to integrals.
Many times an integral depend on some parameter and so the resulting integration is a function in that parameter. For instance take the elliptic integral (\ref{boaa}) and 
assume that $P$ and $Q$ depends on the 
parameter 
$t$ and the interval $\delta$ does not depend on $t$. In any course in calculus we learn that the integration and derivation with respect to $t$ 
commute:
$$
\frac{\partial}{\partial t}\int_{\delta} \frac{Q(x)}{\sqrt{P(x)}}dx=\int_{\delta} \frac{\partial}{\partial t}(\frac{Q(x)}{\sqrt{P(x)}})dx.
$$ 
As before we know that the right hand side of the above equality can be written as a linear combination of two integrals $\int_{\delta}\omega$ and $\int_{\delta}x\omega$. This is 
the historical origin of the notion of Gauss-Manin connection, that is, derivation of integrals with respect to parameters and 
simplifying the result in terms of integrals which cannot be simplified more. 
For instance, take $P(x)$ as in (\ref{8.11.2010}) which depends on three parameters $t_1,t_2,t_3$. We have 
\begin{equation}
\label{sadbis}
\frac{\partial}{\partial t_i}I=I A_i,\ \ i=1,2,3,  \ \hbox{ where } 
I:=[\int_{\delta}\frac{dx}{y}\int_\delta\frac{xdx}{y}],
\end{equation}
and $A_i$ is a $2\times 2$ matrix whose coefficients can be calculated effectively.
When the integrand depends on many parameters the best way to put the information of derivations with respect to all parameters in one object is by using differential forms (recall that differential forms are also used to represent the integrand). 
We define
$$
A=A_1dt_1+A_2dt_2+A_3dt_3
$$
that is, $A_{ij}:=(A_1)_{ij}dt_1+(A_2)_{ij}dt_2+(A_3)_{ij}dt_3,\ i,j=1,2$.  Now, we write (\ref{sadbis}) in the form $dI=I\cdot A$.  In this section we calculate the Gauss-Manin connection, that is, we calculate the matrix $A$, see Proposition \ref{18.1.06}.

\subsection{Gauss-Manin connection}
What we do in this section in the framework of Algebraic Geometry is as follows: 
Let $X$ be a smooth reduced variety over $\ring$. We construct a connection
$$
\nabla: H_\dR^i(X)\to \Omega^1_T\otimes_{\ring} H_\dR^i(X)
$$
where $\Omega_T^1$ is by definition $\Omega_{\ring/\k}^1$, that is, the $\ring$-module of differential attached to $\ring$. 
By definition of
a connection, $\nabla$ is $\k$-linear and satisfies the Leibniz rule
$$
\nabla(r\omega)=dr\otimes \omega+r\nabla\omega.
$$
A vector field $v$ in $T$ is an $\ring$-linear map $\Omega_T^1\to \ring$. We define
$$
\nabla_v:  H_\dR^i(X)\to H_\dR^i(X) 
$$
to be $\nabla$ composed with 
$$
v\otimes {\rm Id}: \Omega^1_T\otimes_{\ring} H_\dR^i(X)\to
\ring\otimes_{\ring} H_\dR^i(X)=H_\dR^i(X).
$$
 If $\ring$ is a polynomial ring 
$\Q[t_1,t_2,\ldots]$ then we have vector fields $\frac{\partial }{\partial t_i}$ which are 
defined by
the rule 
$$
\frac{\partial }{\partial t_i}(dt_j)=1 \text{ if $i=j$ and $=0$ if $i\not=j$}.
$$ 
In this case we simply 
write $\frac{\partial }{\partial t_i}$ instead of $\nabla_\frac{\partial }{\partial t_i}$.

Sometimes it is useful to choose   a basis $\omega_1,\omega_2,\ldots,\omega_h$ of the $\ring$-modular $H^i(X/T)$ and 
write the Gauss-Manin connection in this basis:
\begin{equation}
\label{18oct2010}
\nabla\begin{pmatrix}\omega_1 \\ \omega_2 \\ \vdots \\ \omega_h\end{pmatrix}=A\otimes \begin{pmatrix}\omega_1 \\ \omega_2 \\ \vdots \\ \omega_h\end{pmatrix}
\end{equation}
where $A$ is a $h\times h$ matrix with entries in $\Omega^1_T$.

\subsection{Construction}
Recall the notation of \S \ref{16may2010}. . 
Let us take a covering ${\cal U}=\{U_i\}_{i\in I}$ of $X$ by affine open sets and $\omega\in H^k_\dR(X)$. 
By our definition $\omega$ is represented  by $\oplus_{i=0}^k\omega_i,\ \omega_i\in \Omega^{i}_{k-i}$ and $\omega_{i}$ is  
a collection of $i$-forms $\{\omega_{i,\sigma}\}_\sigma$. By definition we have $d'\omega=0$, where $d'$ is given by 
(\ref{piorenemigo}). Recall that the 
differential map $d$ used in the definition of $d'$ is relative to $\ring$, that is, by definition $dr=0, r\in\ring$. 
Now, let us consider $d$  in the double complex (\ref{geble}) relative to $\k$ 
and not $\ring$.  The condition $d'\omega=0$ turns to be 
$$
d'\omega=\eta, \ \eta=\oplus_{i=0}^{k+1}\eta_i \in {\cal L}^{k+1},\ \eta_i\in 
\Omega^{i}_{k+1-i}
$$ 
and each $\eta_i, \ i\not=0$ is a collection of $i$-forms that is  a finite sum of the form  
$\sum_{j}dr_j\otimes \omega_j$. Since $d$ has no contribution in $\Omega^0_{k+1}$, we know that $\eta_0=0$. We make the tensor 
product $\otimes_{\ring}$ of the double complex (\ref{geble}) with $\Omega^1_T$ and finally  get an element in 
$\Omega^1_T\otimes_{\ring} H_\dR^k(X)$. Of course 
we have to verify that everything is well-defined.

 Let us now assume that $\k=\C$. The main motivation, which is also the historical one, for defining the Gauss-Manin connection is the following:
For any $\omega\in H_\dR^i(X)$ and a continuous family of cycles $\delta_t\in H_i(X_t,\Z)$ we have
\begin{equation}
\label{14.10.10}
d\left (\int_{\delta_t}\omega\right )=\int_{\delta_t}\nabla\omega.
\end{equation}
Here, by definition 
$$
\int_{\delta_t}\alpha\otimes \beta=\alpha\int_{\delta_t}\beta,
$$
where $\beta\in H_\dR^i(X)$ and $\alpha\in \Omega^1_T$. Integrating both side of the equality (\ref{18oct2010}) over a a 
basis $\delta_1,\delta_2,\ldots,\delta_h\in H_i(X_t,\Q)$ we conclude that
\begin{equation}
\label{18.10.10}
d([\int_{\delta_j}\omega_i])=[\int_{\delta_j}\omega_i]\cdot A.
\end{equation}

\begin{exe}\rm
 Discuss in more detail the above construction.
\end{exe}
\subsection{Gauss-Manin connection of families of elliptic curves}
Let us consider the ring $\ring=\Q[t_1,t_2,t_3, \frac{1}{\Delta}]$ and the family of 
elliptic curves in \S \ref{affine}. By Proposition \ref{30may10} we know that the first
de Rham cohomology of $E$ is isomorphic to the first de Rham cohomology of affine variety 
$U_0$.  Therefore, we calculate the Gauss-Manin connection attached to $U_0$. 

\begin{prop}
\label{18.1.06}
The Gauss-Manin connection of the family of elliptic curves $y^2=4(x-t_1)^3-t_2(x-t_1)-t_3$ written in the basis $\frac{dx}{y},\ \frac{xdx}{y}$ is given
as bellow:
\begin{equation}
\label{4mar}
 \nabla\begin{pmatrix}\frac{dx}{y}\\ \frac{xdx}{y}\end{pmatrix}=
A
\begin{pmatrix}\frac{dx}{y}\\ \frac{xdx}{y}\end{pmatrix}
\end{equation}
where 
$$
A= \frac{1}{\Delta}\mat
{-\frac{3}{2}t_1\alpha-\frac{1}{12}d\Delta}
{\frac{3}{2}\alpha}
{\Delta dt_1-\frac{1}{6}t_1d\Delta-(\frac{3}{2}t_1^2+\frac{1}{8}t_2)\alpha}
{\frac{3}{2}t_1\alpha+\frac{1}{12}d\Delta},\
$$
$$
 \Delta=27t_3^2-t_2^3,\ \alpha=3t_3dt_2-2t_2dt_3.
$$
\end{prop}
\begin{proof}
The proof is a mere calculation which is  classical and can be  
found in (\cite{sas74} p. 304, \cite{sai01} ).  We explain only the calculation of
$\frac{\partial}{\partial t_3}(\frac{dx}{y})$ .
For $p(x)=4t_0(x-t_1)^3-t_2(x-t_1)-t_3$ we have:
$$
\Delta=-p'\cdot a_1+p\cdot a_2,
$$ 
where 
$$
a_1=
-36t_0^3x^4+144t_0^3t_1x^3+(-216t_0^3t_1^2+15t_0^2t_2)x^2+(144t_0^3t_1^3-30t_0^2t_1t_2)x-36t_0^3t_1^4+15t_0^2t_1^2t_2-t_0t_2^2
$$
$$
a_2=
(-108t_0^3)x^3+(324t_0^3t_1)x^2+(-324t_0^3t_1^2+27t_0^2t_2)x+(108t_0^3t_1^3-27t_0^2t_1t_2-27t_0^2t_3)
$$ 
We have 
\begin{eqnarray*}
\frac{\partial}{\partial t_3}(\frac{dx}{y}) &= & \frac{-dy\wedge dx}{y^2}=\frac{1}{2}\frac{dx}{py} 
\\ 
&=&  \frac{1}{\Delta}\frac{(-p'a_1+pa_2)dx}{2py}= \frac{1}{\Delta}(\frac{1}{2}a_2-a_1')\frac{dx}{y} \\
&=& (3t_0^2t_1t_2-\frac{9}{2}t_0^2t_3)\frac{dx}{y}
-3t_0^2t_2\frac{xdx}{y}.
\end{eqnarray*}
Note that in the fourth equality above we use $y^2=p(x)$ and the fact that modulo exact 
forms we have
$$
\frac{p'a_1dx}{2py}=\frac{a_1dp}{2py}=\frac{a_1dy}{p}=-a_1d(\frac{1}{y})=\frac{a_1'dx}{y}.
$$

 \end{proof}
\begin{exe}\rm
 Perform all the lacking calculations in Proposition \ref{18.1.06}.
\end{exe}
\subsection{Another family of elliptic curves}
\label{badam}
We modify a little bit the parameter space of our family of elliptic curves. Let
$$
E_\Ra : y^2-4(x-t_1)^3+t_2(x-t_1)+t_3=0,\ (t_1,t_2,t_3)\in T_{\Ra}:=\{(t_1,t_2,t_3)\in \k^3\mid 27t_3^2-t_2^3\not =0\}
$$
be our previous family of elliptic curves and 
\begin{equation}
E_{\rm H}: y^2-4(x-t_1)(x-t_2)(x-t_3)=0, \ (t_1,t_2,t_3)\in T_{\rm H}:=(t_1,t_2, t_3)\in \k^3\mid t_1\not= t_2\not =t_3\}
\end{equation}

The algebraic morphism $\alpha:T_{\rm H}\to T_{\Ra}$ defined by
$$
\alpha: (t_1,t_2,t_3)\mapsto (T, 4\sum_{1\leq i<j\leq 3}(T-t_i)(T-t_j), 4(T-t_1)(T-t_2)(T-t_3)),
$$
where
$$
\ T:=\frac{1}{3}(t_1+t_2+t_3),
$$
connects two families, that is, if in $E_{\Ra}$ we replace $t$ with $\alpha(t)$ we obtain the family $E_{\rm H}$.
The Gauss-Manin connection matrix associated to $E_{\rm H}$ is just the pull-back of the Gauss-Manin connection
associate to $E$. In this way we obtain
\begin{equation}
\label{sapongachibegambehet}
\alpha^{*}A_{\rm R}=A_{\rm H}=
  \frac{dt_1}{2(t_1-t_2)(t_1-t_3)}\mat{-t_1}{1}{t_2t_3-t_1(t_2+t_3)}{t_1}+
\end{equation}
$$
\frac{dt_2}{2(t_2-t_1)(t_2-t_3)}\mat{-t_2}{1}{t_1t_3-t_2(t_1+t_3)}{t_2}+
\frac{dt_3}{2(t_3-t_1)(t_3-t_2)}\mat{-t_3}{1}{t_1t_2-t_3(t_1+t_2)}{t_3}.
$$
\begin{exe}\rm
 Verify the equality (\ref{sapongachibegambehet}).
\end{exe}

\section{Modular differential equations}
\subsection{Introduction}
\label{int03}
After calculation of the Gauss-Manin connection of families of elliptic curves, we immediately calculate the Ramanujan and Darboux-Halphen 
differential equation (Proposition \ref{ezdevaj}).  
The history of these differential equations is full of rediscoveries. The first example of differential equations which has a particular solution given 
by theta constants was studied by Jacobi in 1848.  Later, in 1978 G. Darboux studied the system of differential equations
\begin{equation}
 \label{darboux}
\left \{ \begin{array}{l}
\dot t_1+\dot t_2=2t_1t_2\\
\dot t_2+\dot t_3=2t_2t_3\\
\dot t_1+\dot t_3=2t_1t_3
\end{array} \right.,
\end{equation}
in connection with  triply orthogonal 
surfaces in $\R^3$(see \cite{da78}). G. Halphen (1881), M. Brioschi (1881), 
 and J. Chazy (1909) contributed to the study of the differential equation (\ref{darboux}). In particular, Halphen 
expressed a solution of the system (\ref{darboux}) in terms of the logarithmic derivatives 
of the null theta functions, see \S\ref{halphensection}. Halphen generalized also (\ref{darboux}) to a differential equation with 
three parameters corresponding to the three parameters of the Gauss hypergeometric function. 
His method of calculating such differential equations is essentially described in \S\ref{perram} and it is near in spirit to 
the methods used in the present text.  
The history from number theory point of view is different. S. Ramanujan, who was a master of convergent and formal series and who did not know about the geometry of 
differential equations, in 1916 observed that the derivation of three Eisenstein series $E_2,E_4$ and $E_6$ are polynomials 
in $E_i$'s. Therefore, he had the solution and he found the corresponding differential equation. This is opposite to the work of Halphen who had the differential equation and 
he calculated a solution (the later in general is more difficult than the former). 
Halphen even calculated the Ramanujan differential equation years before Ramanujan and apparently without knowing about Eisenstein series (see \cite{hal00} page 331). 
It is remarkable to say that Darboux-Halphen differential equations was rediscovered in mathematical physics by M. Atiyah and N. Hitchin in 1985. Even
the author of the present text calculated Ramanujan and Darboux-Halphen differential equations independently without knowing about Ramanujan, Darboux and Halphen's work. All these rediscoveries  were useful because they have helped us to understand the importance and applications of such differential equations and also to find the general context of such differential equations in related with Gauss-Manin connections and in general multi dimensional linear differential systems.

\subsection{Ramanujan vector field}
\label{uff}
Our main observation in this section is the following:

\begin{prop}
\label{ezdevaj}
In the parameter space of the family of elliptic curves 
 $y^2=4(x-t_1)^3-t_2(x-t_1)-t_3$ there is a unique vector field $\Ra$, such that
\begin{equation}
\label{niteroi}
\nabla_{\Ra}(\frac{dx}{y})= -\frac{xdx}{y},\ \nabla_{\Ra}(\frac{xdx}{y})= 0.
\end{equation}
The vector field $\Ra$ is given by
\begin{equation}
\label{ramanve}
\Ra=(t_1^2-\frac{1}{12}t_2)\frac{\partial}{\partial t_1}+ 
(4t_1t_2-6t_3)\frac{\partial}{\partial t_2}+ 
(6t_1t_3-\frac{1}{3}t_2^2)\frac{\partial}{\partial t_3}.
\end{equation}
\end{prop}
\begin{proof}
 The proof is based on explicit calculations.
\end{proof}
\begin{exe}\rm
 Perform the calculations leading to a proof of Proposition \ref{ezdevaj}.
\end{exe}

\subsection{Vector field or ordinary differential equation?}
\label{ramode}
Any vector field in $\C^n$ represent an ordinary differential equation, for which we can study the dynamics of its solutions.
For instance, the vector field $\Ra$ of the previous section can be seen as the following
ordinary differential equation:
\begin{equation}
\label{raman}
{\rm R}:
 \left \{ \begin{array}{l}
\dot t_1=t_1^2-\frac{1}{12}t_2 \\
\dot t_2=4t_1t_2-6t_3 \\
\dot t_3=6t_1t_3-\frac{1}{3}t_2^2
\end{array} \right.
\end{equation} 
where dot means derivation with respect to a variable. For a moment forget all what we have done 
to obtain $\Ra$. We write each $t_i$ as a formal power series in $q$,
$t_i=\sum_{n=0}^\infty t_{i,n}q^n,\ i=1,2,3$ and substitute in the above differential 
equation. We define the derivation to be:
$$
\dot t=aq\frac{\partial }{\partial q}
$$
for a fixed non zero number $a$. Comparing the coefficients of $q^0$ we have
$$
(t_{1,0},t_{2,0},t_{3,0})=(b,12b^2,8b^3),\hbox{ for some } b\in\C.
$$
Comparing the coefficients of $q^1$ we have 
$$
MV=aV,\ \hbox{ where },\ \
M:=\begin{pmatrix}
    2b&-\frac{1}{12}& 0\\
    48b^2&4b&-6\\
    48b^3&-8b^2&6b
   \end{pmatrix},\ 
 V:=\begin{pmatrix}
     t_{1,1}\\
t_{2,1}\\
t_{3,1}
    \end{pmatrix}.
$$
Assume that $V$ is not zero and so we have $\det(M-aI_{3\times 3})=(12b-a)a^2=0$ which implies that  $a=12b$. We also calculate $V$ up to multiplication
by a constant:
$$
V^\tr=c[-24b, 240(12b^2), -504(8b^3)],\ \hbox{ for some } c\in\C.
$$  
We realize that all the coefficients $t_{i,n},\ n>1$ are determined uniquely and recursively:
$$
(12nb I_{3\times 3}-M)
\begin{pmatrix}
t_{1,n}\\
t_{2,n}\\
t_{3,n}
    \end{pmatrix}=
\begin{pmatrix}
\sum_{i=1}^{n-1}t_{1,i}t_{1,n-i}\\
4\sum_{i=1}^{n-1}t_{1,i}t_{2,n-i}  \\
\sum_{i=1}^{n-1}6t_{1,i}t_{3,n-i}-\frac{1}{3}t_{2,i}t_{2,n-i}
    \end{pmatrix}
$$
\begin{prop}
\label{20102010}
We have $t_k=a_kE_k(cq)$, where
\begin{equation}
\label{eisenstein}
E_k(q):=\left (1+b_k\sum_{n=1}^\infty \left (\sum_{d\mid n}d^{2k-1}\right )q^{n}\right ),\ \  k=1,2,3,
\end{equation}
are Eisenstein series and 
$$
(b_1,b_2,b_3)=
(-24, 240, -504),\ \ (a_1,a_2,a_3)=(b,12 b^2 ,8b^3).
$$
\end{prop}
In fact what Ramanujan did was to verify that the Eisenstein series $E_k,\ k=1,2,3$ satisfy an ordinary differential equation 
which is obtained from (\ref{raman}) after an affine transformation 
$(t_1,t_2,t_3)\mapsto (\frac{1}{12}t_1,\frac{1}{12}t_2,\frac{2}{3(12)^2}t_3)$,  see \cite{nes01}, p. 4 or \cite{maro05}. 

Later, we will see that $t_{k}$'s are convergent in the unit disc. If we set $q=e^{2\pi i z}$ and look at $t_k$'s as holomorphic functions in $z$, 
which varies in the upper half plane $\{z\in\C \mid \Im(z)>0\}$, then they are basic examples of quasi-modular forms. Note that for $b=\frac{1}{12}$ and $a=1$ we have 
$2\pi i q\frac{\partial}{\partial q}=\frac{\partial}{\partial z}$ and so if we define 
\begin{equation}
 \label{17nov2010}
(g_1(z),g_2(z),g_3(z))=(\frac{2\pi i}{12} E_2, 12(\frac{2\pi i}{12})^2 E_4, 8(\frac{2\pi i}{12})^3 E_6)
\end{equation}
 then $g_k$'s satisfy the Ramanujan
differential equation (\ref{raman}), where dot means derivation with respect to $z$.

\begin{exe}\rm
\label{presidenta}
 \begin{enumerate}
  \item
For a moment forget Proposition \ref{20102010}. 
Let $t_i=\sum_{n=0}^\infty t_{i,n}q^n,\ i=1,2,3$ be formal power series in $q$ with unknown coefficients as before and assume that $t_i$ satisfy the 
differential equation \ref{raman}, where dot is $12q\frac{\partial}{\partial q}$ and we know the initial value $t_{1,2}=-24$.  
Calculate the first five coefficients $t_{k,n},\ k=1,2,3,\ n=1,2,\ldots 5$ and compare them with the coefficients in the Eisenstein series. 
\item 
Find a proof for (\ref{20102010}) in the literature (see \cite{nes01, maro05}).
 \end{enumerate}

\end{exe}

%

\subsection{Halphen vector field}
\label{halphensection}
Let us consider the family of elliptic curves considered in \S\ref{badam}.
A vector field with the properties (\ref{niteroi}) is given by
$$
{\rm H}=
(t_1(t_2+t_3)-t_2t_3)\frac{\partial}{\partial t_1}+ 
(t_2(t_1+t_3)-t_1t_3)\frac{\partial}{\partial t_2}+ 
(t_3(t_1+t_2)-t_1t_2)\frac{\partial}{\partial t_3}. 
$$
The corresponding ordinary differential equation is
\begin{equation}\rm
{\rm H}:\label{halphen}
\left \{ \begin{array}{l}
\dot t_1=t_1(t_2+t_3)-t_2t_3\\ 
\dot  t_2= t_2(t_1+t_3)-t_1t_3 \\
 \dot  t_3= t_3(t_1+t_2)-t_1t_2
\end{array} \right.
\end{equation}
which is the same as (\ref{darboux}). Halphen expressed a solution of the system (\ref{halphen}) in terms of the logarithmic 
derivatives 
of the null theta functions; namely, 
\begin{prop}
\label{25002500}
 The holomorphic functions
$$
u_1=2(\ln \theta_4(0|z))',
u_2=2(\ln \theta_2(0|z))',
u_3=2(\ln \theta_3(0|z))'
$$
where
$$
\left \{ \begin{array}{l}
\theta_2(0|z):=\sum_{n=-\infty}^\infty q^{\frac{1}{2}(n+\frac{1}{2})^2}
\\
\theta_3(0|z):=\sum_{n=-\infty}^\infty q^{\frac{1}{2}n^2}
\\
\theta_4(0|z):=\sum_{n=-\infty}^\infty (-1)^nq^{\frac{1}{2}n^2}
\end{array} \right., 
\ q=e^{2\pi i z},\  z\in \uhp.
$$
satisfy the ordinary differential equation (\ref{halphen}), where 
the derivation is with respect to $z$.
\end{prop}

\begin{exe}\rm
\begin{enumerate}
\item
Find a proof of Proposition \ref{25002500}, see \cite{ha81, ohy95}.
\end{enumerate}
\end{exe}
\subsection{Relations between theta and Eisenstein series}
As we mentioned in \S\ref{badam}, the map $\alpha: T_{\rm H}\to T_{\Ra}$ maps the Gauss-Manin connection matrix of $E_{\rm H}$ to the Gauss-Manin connection
matrix of $E_{\Ra}$, both written in the basis $\frac{dx}{y},\ \frac{xdx}{y}$. We know  that $X=\Ra$ or ${\rm H}$ are both determined uniquely 
by $\nabla_X \frac{dx}{y}=-\frac{xdx}{y},\ \nabla_X \frac{xdx}{y}=0$. All these imply that the vector field $\rm H$ is mapped to the vector field $\Ra$
through the map $\alpha$. Let $u:\uhp\to\C^3,\ u(z)=(u_1(z),u_2(z),u_3(z))$ be the solution of $\rm H$ by the logarithmic derivative of theta functions. 
It turns out that $\alpha (u(z))$ is a solution of $\Ra$, and we claim that it is $(g_1,g_2,g_3)$ defined in (\ref{17nov2010}). 
Two solutions of (\ref{raman}) with the same $t_1$ coordinates are
equal (this follows by the explicit expression of (\ref{raman})). From this and the discussion in \S\ref{ramode} it follows that it is enough to prove that both 
$\frac{1}{3}(u_1+u_2+u_3)$ and $g_1(z)$ have the form  $\frac{2\pi i}{12}(1-24q+\cdots)$.
Finally, we get the equalities:
$$
\frac{2}{3}\ln(\theta_2(0|z)\theta_3(0|z)\theta_4(0|z))'=\frac{2\pi i}{12}E_2 (z),
$$
$$
 4\cdot 4\sum_{2\leq i<j\leq 4}
\ln(\frac{(\theta_2(0|z)\theta_3(0|z)\theta_4(0|z))^{\frac{1}{3}}}{\theta_i(0|z)})'
\ln(\frac{(\theta_2(0|z)\theta_3(0|z)\theta_4(0|z))^{\frac{1}{3}}}{\theta_j(0|z)})'=12(\frac{2\pi i}{12})^2E_4 (z),
$$
$$
4\cdot 8 \prod_{i=1}^3\ln(\frac{(\theta_2(0|z)\theta_3(0|z)\theta_4(0|z))^{\frac{1}{3}}}{\theta_i(0|z)})'
=8(\frac{2\pi i}{12})^3E_6 (z).
$$
\subsection{Automorphic properties of the special solutions}
Let us define
$$
\SL 2\Z:=\{\mat{a}{b}{c}{d}\mid a,b,c,d\in\Z,\ ad-bc=1\}
$$
and\index{$\SL 2\Z$, full modular group }
$$
\Gamma(d):=\{A\in\SL 2\Z\mid A\equiv \mat{1}{0}{0}{1} {\rm mod } \ d\},\ d\in\N.
$$\index{$\Gamma(d)$, modular group}
For a holomorphic function $f(z)$ let  also  define:
$$
(f\mid_m^0A)(z):=(cz+d)^{-m}f(Az),\ (f\mid_m^1A)(z):=
(cz+d)^{-m}f(Az)-c(cz+d)^{-1},\
$$
$$ A=\mat{a}{b}{c}{d}\in \SL 2\C,\ m\in\N.
$$ 
\begin{exe}\rm
If $\phi_i,\ i=1,2,3$  are the coordinates of 
a solution of $\rm R$(resp. $\rm H$) then 
$$
\phi_1\mid_2^1A,\ \phi_2\mid_4^0A,\ \phi\mid_6^0A
$$
(resp.
$$
\phi_i\mid_2^1A,\  i=1,2,3,
$$)
are also coordinates of a solution of $\rm R$ (resp. $\rm H$) for all
$A\in\SL 2\C$. 
The subgroup
of $\SL 2\C$ which fixes the solution given by Eisenstein series
(resp. theta series) is $\SL 2\Z$ (resp. $\Gamma(2)$) (the second part of the exercise will be verified in \S\ref{ellipticsection}.
\end{exe}

\subsection{Another example}  
Let 
$$
\eta(z):=q^{\frac{1}{24}}\prod_{n=1}^\infty(1-q^n),\ q=e^{2\pi i z}
$$
be the Dedekind's $\eta$-function. 
In \cite{oh01} Y. Ohyama has found that 
\begin{align}
W&=(3\log\eta(\frac{z}{3})-\log\eta(z))' \\
X&=(3\log\eta(3z)-\log\eta(z))' \\
Y&=(3\log\eta(\frac{z+2}{3})-\log\eta(z))' \\
Z&=(3\log\eta(\frac{z+1}{3})-\log\eta(z))' 
\end{align}
satisfy the equations:
$$
\left\{
\begin{array}{l}
t_1'+t_2'+t_3'=t_1t_2+t_2t_3+t_3t_1\\
t_1'+t_3'+t_4' =t_1t_3+t_3t_4+t_4t_1\\
t_1'+t_2'+t_4'=t_1t_2+t_2t_4+t_4t_1\\
t_2'+t_3'+t_4'=t_2t_3+t_3t_4+t_4t_2\\
\zeta_3^2(t_2t_4+t_3t_1)+\zeta_3(t_2t_1+t_3t_4)+(t_2t_3+t_4t_1)=0
\end{array}
\right.,
$$
where $\zeta_3=e^{\frac{2\pi i}{3}}$. We write the first four lines
of the above equation as a solution to a vector field $V$ in 
$\C^4$ and let $F(t_1,t_2,t_3,t_4)$ be the polynomial in the fifth line.
Using a computer, or by hand if we have a good patience for
calculations, we can verify the equality $dF(V)=0$ and so $V$ is tangent to 
$T:=\{t\in \C^4\mid F(t)=0\}$.
Our  discussion leads to:
\begin{exe}\rm
\label{6oct2010}
Show that there is a family of elliptic curves $E\to T$ and a basis 
$\omega_1,\omega_2\in H^1_\dR(E/T)$ such that $\omega_1$ is a regular differential formal and
$$
\nabla_\Ra(\omega_1)=-\omega_2,\ \ \nabla_\Ra(\omega_2)=0,
$$
where $\Ra$ is the restriction of $V$ to $T$. 
\end{exe}

\section{Weierstrass form of elliptic curves}
\label{weierstrasssection}
\subsection{Introduction}
\label{int04}
We can think of an elliptic curve over rational numbers as the Diophantine equation
$$
y^2=4P(x),
$$
where $P$ is a monic degree three polynomial in $x$ with rational coefficients and without double roots in $\C$. In fact, this is the Weierstrass form of any elliptic
curve in the framework of algebraic geometry. The moduli of elliptic curves is one dimensional and we have apparently three independent parameters in the polynomial $P$. It turns out
that there is an algebraic group of dimension two  which acts on the coefficients space of $P$ and the resulting quotient is the moduli of elliptic curves, see \S\ref{9nov10}. 
In this section, we remark that the coefficients space of $P$ has also a moduli interpretation.  We consider elliptic curves with elements in their
algebraic de Rham cohomologies and we ask for normal forms of such objects. It turns out that the three independent coefficients of the polynomial $P$ are the coordinates system of the moduli
of such enhanced elliptic curves.

\subsection{Elliptic curves} 

Let $\k$ be a field of characteristic zero. For a reduced smooth curve $C$ over $\k$ we define its genus to be the dimension of the space of regular differential forms on $C$.
\begin{defi}\rm
An elliptic curve over $\k$ is a pair $(E,O)$, where $E$ is a genus one complete smooth curve and 
$O$  is a $\k$-rational point 
of $E$. 
\end{defi}
Therefore, by definition an elliptic curve over $\k$ has at least a $\k$-rational point. A smooth
projective curve of degree $3$ is therefore an elliptic curve if it has a $\k$-rational point.
For instance, the Fermat curve
$$
F_3: x^3+y^3=z^3
$$
is an elliptic curve over $\Q$. It has $\Q$-rational points $[0;1;1]$ and $[1;0;1]$. However 
$$
E: 3x^3+4y^3+5z^3=0
$$ 
has not $\Q$-rational points and so it  is not an elliptic curve defined over $\Q$. It is an interesting fact to
mention that $E(\Q_p)$ for all prime $p$ and $E(\R)$ are not empty. 
This example is due to Selmer (see \cite{cas66, sel51}).

\subsection{Weierstrass form}
In this section we prove the following proposition:
\begin{prop}
\label{7.3.8}
 Let $E$ be an elliptic curve over a field $\k$ of characteristic $\not=2,3$ and let $\omega$ be a regular differential 
form on $E$. There exist unique functions $x,y\in \k(E)$ such that the map
$$
E\to \P ^2,\ a\mapsto [x(a);y(a);1]
$$
gives an isomorphism of $E$ onto a curve given by 
$$
y^2=4x^3-t_2x-t_3,\ t_2,t_3\in \k
$$
sending $O$ to $[0;1;0]$, and $\omega=\frac{dx}{y}$. 
\end{prop}
We call $x$ and $y$ the Weierstrass coordinates of of $E$.
\begin{proof}
For a divisor $D$ on a curve $C$ over $\bar \k$ define the linear system
$$
\L(D)=\{f\in \bar \k(C),\ f\not=0 \mid div(f)+D\geq 0\}\cup \{0\}
$$
and
$$
l(D)=\dim_{\bar \k} (\L(D)).
$$
We know by Riemann-Roch theorem that
$$
l(D)-l(K-D)=\deg(D)-g+1,
$$
where $K$ is the canonical divisor of $C$. We have $\deg(K)=2g-2$ and so for $\deg(D)>2g-2$, equivalently $\deg(K-D)<0$,
 we have
$$
l(D)=\deg(D)-g+1.
$$ 
For $g=1$ and $D=nO$ we get $l(D)=n$. Using this for $n=2,3$, we can choose $x,y\in \k(E)$ such that 
$1,x$ form a basis of $\L(2O)$ and $1,x,y$ form a basis of $\L(3O)$. N
The function $x$ (resp. $y$) has a pole of order $2$ (resp. $3$) at $O$. In fact, we need the following choice of $y$, $y:=\frac{dx}{\omega}$. 
Note that $\omega$ is regular and vanishes nowhere. The map $\sigma: E\to E,\ P\mapsto -P$ acts trivially on any $x$ in $\L(2O)$ because $\sigma x-x$ has a simple pole, and hence by residue formula for $(\sigma x-x)\omega$, has no pole and so $\sigma x=x$. 
We have $l(6O)=6$ and so there is a linear relation between $1,x,x^2,x^3,y^2,y,xy$. The last two terms $y, xy$ does not appear on such a linear relation: 
a point $p$ is a non-zero 2-torsion point of $E$ if and only if $p$ is a double root of $x-x(p)$ and if and only if $y(p)=0$. 
We can further assume that the coefficient of $y^2$ is one and of $x^3$ is $4$. After a substitution of $x$ with $x+a$ for some $a\in\k$, we get the desired polynomial 
relation between $x$ and $y$. 

\end{proof}
\begin{prop}
\label{manga2010}
Let $E$ be an elliptic curve over a field $\k$ of characteristic zero and 
$\omega\in H^1_\dR(E)\backslash F^1$. 
There exist unique functions $x,y\in \k(E)$ such that the map
$$
E\to \P ^2,\ a\mapsto [x(a);y(a);1]
$$
gives an isomorphism of $E/k$ onto a curve given by 
$$
y^2=4(x-t_1)^3-t_2(x-t_1)-t_3,\ t_1,t_2,t_3\in \k
$$
sending $O$ to $[0;1;0]$ and $\omega=\frac{xdx}{y}$. 
\end{prop}
\begin{proof}
 We have a regular differential form $\omega_1$ on $E$ which is determined uniquely by 
$\langle \omega_1,\omega\rangle=1$. We apply Proposition \ref{7.3.8} for the pair $(E,\omega_1)$. In the corresponding 
Weierstrass coordinates by Proposition \ref{redbull} we can write $\omega=t_1\frac{dx}{x}+t_0\frac{xdx}{y},\ t_1,t_0\in\k$. 
Using (\ref{08.05.2010}) and 
 $\langle \omega_1,\omega\rangle=1$, we have $t_0=1$. We now substitute $x$ by $x-t_1$. 
\end{proof}
\begin{exe}\rm
 Prove that the map $E\to \P ^2$ in Proposition \ref{7.3.8} and Proposition \ref{manga2010} gives an isomorphism of $E$ onto its image. 
\end{exe}


\subsection{Group structure}
From now on will drop $O$ and simply write $E$ instead of $(E,O)$. An elliptic curve carries a structure of an Abelian group. We explain this for a 
smooth cubic in $\P ^2$.

 Let $E$ be a smooth cubic curve in $\P^2$ and $O\in E(\k)$.
 Let also $P,Q\in E(\k)$ and $L$ be the line in $\P^2$ 
connecting two points $P$ and $Q$. If $P=Q$ then $L$ is the tangent line to $E$ at $P$.
 The line $L$ is defined over $\k$ and it is easy to verify that 
the third intersection $R:=PQ$ of $E(\bar \k)$ with $L(\bar \k)$ is also in $E(\k)$. Define
$$
P+Q=O(PQ)
$$
For instance, for an elliptic curve in the Weierstrass form take $O=[0;1;0]$ the point at infinity. By definition
$O+O=O$. The above construction turns $E(\k)$ into a commutative group with the zero element $O$.

\begin{exe}\rm
\begin{enumerate}
 \item 
Let $E$ be an elliptic curve over $\k$, $a\in E(\k)$ and $f:E\to E,\ f(x)=x+a$. 
Prove that the induced map in the de 
Rham cohomology is identity. 
\item
For  $g:E\to E,\ g(x)=nx,\ n\in\N$, prove that the induced map in the de Rham cohomology is 
multiplication by $n$. In fact in the Weierstrass coordinates, and using Weierstrass $p$ 
function, we have the equality:
$$
f^*(\frac{xdx}{y})=(\frac{1}{4}(\frac{y-y_0}{x-x_0})^2-x-x_0)\frac{dx}{y},\ \ a=(x_0,y_0).
$$
This is a real equality between differential forms and not modulo exact forms.
\end{enumerate}
\end{exe}

\subsection{Moduli spaces of elliptic curves}
\label{modulisection}
Let $\ring=\k[t_1,t_2,t_3, \frac{1}{\Delta}], \ \Delta:=27t_3^2-t_2^3 $  and 
$T=\spec(\ring)$. Let also $E$ be the subvariety of $\P ^2\times T$ given by:
\begin{equation}
\label{khodaya}
E: zy^2-4(x-t_1)^3+t_2z^2(x-t_1)+t_3z^3=0, [x;y;z]\times (t_1,t_2,t_3)\in\P ^2\times T
\end{equation}
and 
$$
E\to T
$$
be the projection on $T$. The differential forms $\frac{dx}{y}, \frac{xdx}{y}$ 
form a free basis of the $\ring$-module $H_\dR^1(E/T)$, $\frac{dx}{y}$ is a regular differential form on $E$, 
 and $\langle \frac{dx}{y}, \frac{xdx}{y}\rangle=1$.
The discussion in \S \ref{weierstrasssection} leads to
\begin{prop}
\label{sarabishoghl}
 The affine variety $T$ is the moduli of the pairs $(\tilde E,\omega)$, where $\tilde E$ is 
an elliptic curve over $\k$ and $\omega\in H_\dR^1(\tilde E)\backslash F^1$.
\end{prop}
Note that from the beginning we could work with the elliptic curve $E$ in the Weierstrass form with $t_1=0$. We have the  isomorphism
$$
(\{y^2=4(x-t_1)^3-t_2(x-t_1)-t_3\}, \frac{xdx}{y})\cong (\{y^2=4x^3-t_2x-t_3\}, \frac{xdx}{y}+ t_1 \frac{dx}{y}),
$$
$$
(x,y)\mapsto (x-t_1,y).
$$
\subsection{Torsion points}
For an elliptic curve over $\k$ we define $E[N]$ to be the set of $N$-torsion points of $E$:
$$
E[N](\k):=\{p\in E(\k)\mid Np=0\}.
$$
When the base field is clear from the text, we drop $(\k)$ and simply write 
$E[N]=E[N](\k)$. 
\begin{prop}
\label{torsiongroup}
 Let $E$ be an elliptic curve over $\k$. We have an isomorphism of groups
$$
E[N](\bar \k)\cong (\Z/N\Z)^2
$$
and so the cardinality of $E[N](\k)$ is less than $N^2$.
\end{prop}
For a proof see \cite{si86}, Theorem 6.1 page 165.

\section{Quasi-modular forms}
\label{qmfsection}
\subsection{Introduction}
\label{int05}
In the present section we introduce quasi-modular forms in the framework of Algebraic Geometry.
For an elliptic curve with an element in its de Rham cohomology, which is not represented by a regular differential form,  we can associate three quantities $t_1,t_2,t_3$ which
appear in Weierstrass form of $E$, see Proposition \ref{manga2010}. These quantities satisfy a simple functional property
with respect to the action of an algebraic group which turn them our first examples of quasi-modular forms. These are algebraic version
of the Eisenstein series. In fact, we can describe the algebraic version of all Eisenstein series by using residue calculus, see  \S\ref{alleisen}.
For future applications, we introduce quasi-modular forms for the elliptic curves enhanced with certain torsion elements structure. 
 
\subsection{Enhanced elliptic curves}
Let $N$ be a positive integer. 
 An enhanced elliptic curve for $\Gamma_0(N)$ is a 3-tuple $(E, C,\omega)$, where $E$ is an elliptic curve over $\k$, $C$ is a cyclic subgroup of $E(\k)$ of order $N$ and $\omega$ is an element in $H_{\dR}^1(E)\backslash F^1$. 
 An enhanced elliptic curve for $\Gamma_1(N)$ is a $3$-tuple $(E,Q,\omega)$, where $E,\omega$ are as before and $Q$ is a point of $E(\k)$ of order $N$. Let us fix a 
primitive root of unity of order $N$ in $\k$, say $\zeta$.
An enhanced elliptic curve for $\Gamma(N)$ is a $3$-tuple 
$(E,(P,Q),\omega)$, where $E,\omega$ are as before and
 $P$ and $Q$ are  a pair of points of $E(\k)$ that generates the $N$-torsion subgroup $E[N]$ with Weil pairing $e_N(P,Q)=\zeta$. 
For the definition of Weil pairing see Chapter 3, Section 8 of Silverman \cite{si86}.

Note that the choice of $\omega\in H_{\dR}^1(E)\backslash F^1$ determines in a unique 
way an element $\omega_1\in F^1$ ($\omega_1$ is a differential form of the first kind)  with
$\langle \omega_1, \omega\rangle=1$. In this way, $\omega_1,\omega$ form a basis of the 
$\k$-vector space $H_{\dR}^1(E)$.


In a similar way we can define a family of enhanced elliptic curves 
(see \cite{har77}, Chapter III, Section 10).

%
%
%

\subsection{Action of algebraic groups}
\label{9nov10}
Let $\Gamma$ be one of the $\Gamma_0(N),\Gamma_1(N)$ and $\Gamma(N)$ and $T_\Gamma$ be the set of enhanced elliptic curves for $\Gamma$ modulo canonical isomorphisms. The additive group $G_a=(\k,+)$ and the multiplicative group $G_m=(\k^*,\cdot)$  acts in a canonical way on $T_\Gamma$:
$$
(*,*,\omega)\bullet k=(*,*, k^{-1}\omega), \ k\in G_m, \ (*,*,\omega)\in T_\Gamma,
$$
$$
(*,*,\omega)\bullet k=(*,*, k'\omega_1+\omega), \ k'\in G_a, \ (*,*,\omega)\in T_\Gamma.
$$
Both these actions can be summarized in the action of the algebraic group
\begin{equation}
 \label{2nov10}
G=
\left\{\mat{k}{k'}{0}{k^{-1}}\mid \ k'\in \k, k\in \k-\{0\}\right\}\cong G_a\times G_m
\end{equation}
i.e.
$$
(*,*,\omega)\bullet g=(*,*, k'\omega_1+k^{-1}\omega), \ g\in G, \ (*,*,\omega)\in T_\Gamma.
$$

In the case $N=1$ an enhanced elliptic curve becomes a pair 
$(E,\omega),\ \omega\in H^{1}_\dR(E)\backslash F^1$. In \S \ref{modulisection} we constructed
the moduli space of such pairs:
\begin{prop}
\label{actiont}
The canonical map
$$
T(\k)\to T_\Gamma,\ t\mapsto  (E_{\pi^{-1}{t}},\frac{xdx}{y})
$$
is an isomorphism. Under this isomorphism the action of the algebraic group $G$ is given by 
$$
t\bullet g:=(
 t_1k^{-2}+k'k^{-1},
t_2k^{-4}, t_3k^{-6}), t=(t_1,t_2,t_3),
g=\mat {k}{k'}{0}{k^{-1}}\in G.
$$
 \end{prop}
\begin{proof}
 The first part follow from Proposition \ref{sarabishoghl}. The proof of the second part is 
as follows:
We first prove (\ref{gavril}). Let
$$
\alpha: \A^2_\k\rightarrow \A^2_\k,\ (x,y)\mapsto
(k^2x-k'k, k^{3}y)
$$
and $f=y^2-4(x-t_1)^3+t_2(x-t_1)+t_3$. We have
$$
k^{-6}\alpha^{*}(f)=y^2-4k^{-6} (
k^{2}x-k'k-t_1)^3+ t_2k^{-6}(
k^2x-k'k-t_1)+t_3k^{-6}=
$$
$$
y^2-4(x-k'k^{-1}-t_1k^{-2})^3+ t_2k^{-4}(x-k'k^{-1}-t_1k^{-2})+t_3k^{-6}.
$$
This implies that $\alpha$ induces an isomorphism of elliptic curves
$$
\alpha: (E_{t\bullet g}, \alpha^*(\frac{xdx}{y})\rightarrow (E_t,\frac{xdx}{y}).
$$
Since
$$
\alpha^*\frac{xdx}{y}=k \frac{xdx}{y}-k'\frac{dx}{y}
$$
we get the result.
\end{proof}

\subsection{Quasi-modular forms}
A quasi-modular form $f$ of weight $m$ and differential order $n$ for $\Gamma$ is a function $T_\Gamma \to \k$ with the following properties: 
\begin{enumerate}
\item
With respect to the action of $G_m$, $f$ satisfies
\begin{equation}
\label{gmaction}
f\bullet k=k^{-m}f,\ \k\in G_m.
\end{equation}
\item
With respect to the action of $G_a$, $f$ satisfies the following condition: 
there are functions $f_i:T_\Gamma\to \k,\ i=0,1,2,\ldots,n$ such that
\begin{equation}
\label{gaaction}
f\bullet k'=\sum_{i=0}^n\bn ni{k'}^if_i,\ k'\in G_a.
\end{equation}
\item
(Growth condition)
Let $\pi: E\to T$ be a family of elliptic curves over $\k$, with possibly singular fibers, and $S\subset T$ be a subvariety of $T$ such that
the induced map $\tilde E\to \tilde T$, where $\tilde E:=E\backslash \pi^{-1}(S)$ and $\tilde T:=T\backslash S$, is the underlying morphism of an enhanced family, say  
$(\tilde E ,*,\omega)$ with the torsion structure 
$*$ and $\omega\in H^1_\dR(\tilde E/\tilde T)$.   
 The function $f$ induces a map $\tilde T(\k)\to \k$.
We assume the following growth condition for $f$: for $E\to T$ and the torsion structure $*$  fixed as above, there is $\omega$ such that
$\tilde T(\k)\to \k$ extends to a function $T(\k)\to \k$. In other words, we may change $\omega$ in order to guarantee the extension of $\tilde T(\k)\to \k$. 
\end{enumerate}
The variety $S$ contains the loci of all closed points $t\in T(\k)$ such that either 
$E_t$ is  singular or $\omega\mid_{E_t}$ is a differential form
 of the first kind or zero or some degeneracy occurs on the torsion point structure of $E_t$. 

We denote by $\mf{n}{m}(\Gamma)$ the set of quasi-modular forms of weight $m$ and differential order 
$n$ and we set
$$
\mf{}{}(\Gamma)=\sum_{m\in \Z, n\in \N_0} \mf{n}{m}(\Gamma).
$$
When there is no confusion we drop $\Gamma$ and simply write $\mf{}{}=\mf{}{}(\Gamma)$ and so on.
If $n\leq n'$ then $\mf nm\subset \mf {n'}{m}$ and
$$
\mf nm\mf {n'}{m'}\subset \mf {n+n'}{m+m'}, \ \mf nm+\mf {n'}{m}=\mf
{n'}m.
$$
We see that $\mf{}{}$ has a structure of a graded $\k$-algebra, 
the degree of the elements of $\mf{n}{m}$ is $m$.

For $n=0$ we recover the definition of modular forms of weight $m$. A modular form of weight 
$m$ is a function from the set of enhanced elliptic curves as before but with this difference 
that $\omega\in F^1$ is a regular differential 
form and not an element in $H^1_\dR(E)\backslash F^1$. 
The action of $G_m$ is given by  $(*,*,\omega)\bullet k=(*,*, k\omega)$ and $f$ satisfies 
$f\bullet k=k^{-m}f,\ \k\in G_m$ and the growth condition  as above (see Exercise 1 at the end of this section).

Combining both actions 
(\ref{gmaction}) and (\ref{gaaction}) we have:
\begin{equation}
\label{jabr}
f\bullet g=k^{-m}\sum_{i=0}^n\bn ni{k'}^ik^{i}f_i,\ 
\forall g=\mat{k}{k'}{0}{k^{-1}}\in G.
\end{equation}

\begin{exe}\rm
\begin{enumerate}
\item There is a canonical bijection between modular forms of weight $m$  and quasi-modular forms of weight $m$ and differential order $0$. 
\item
Verify that $f_i$ is a quasi-modular form of weight $m-2i$ and differential
 order $n-i$. In particular, $f_n$ is a modular form of weight $m-2n$. 
 \item
 Show that the family mentioned in Exercise \ref{6oct2010} is the universal family of enhanced elliptic curves for $\Gamma(3)$. 
\end{enumerate}
\end{exe}

\subsection{Full quasi-modular forms}

\label{zaifam}
Let us consider again the case $N=1$. In this case a quasi-modular form for $\Gamma$ 
is also called a full quasi-modular form. Using Proposition \ref{actiont}, any full quasi 
modular formal turns out to be a regular function on $\spec(\k[t_1,t_2,t_3])$, that is, a 
polynomial
in $t_1,t_2,t_3$ and with coefficients in $\k$. Using the same theorem, we know that 
$t_i,\ i=1,2,3$ is a quasi-modular form of weight $2i$ and
differential order $1$ for $t_1$ and $0$ for $t_2$ and $t_4$. This proves:
\begin{prop}
\label{domingodedanca}
 The $\k$-algebra of full quasi-modular forms is isomorphic to
$$
\k[t_1,t_2,t_3],\ \deg(t_i)=2t_i,\ i=1,2,3.
$$
\end{prop}
The $\k$-algebra of full quasi-modular forms has also a differential structure which is given 
by:
$$
 \mf{n}{m}\to \mf{n+1}{m+1},\ t\mapsto dt({\Ra})=\sum_{i=1}^3\frac{\partial t}{\partial t_i}
{\Ra}_i
$$
where $\Ra=\sum_{i=1}^3 \Ra_i\frac{\partial}{\partial t_i}$ is the Ramanujan vector field. 
We sometimes
use $\Ra: \mf {}{} \to \mf {}{}$ to denote this differential operator.

In a similar way, the family $y^2-4(x-s_1)(x-s_2)(x-s_3)=0$ is the universal family 
for the moduli of 
 $3$-tuple $(E,(P,Q),\omega)$, where $E,\omega$ are as before and
 $P$ and $Q$ are  a pair of points of $E(\k)$ that generates the $2$-torsion subgroup 
$E[N]$ with Weil pairing $e(P,Q)=-1$. The points $P$ and $Q$ are given by $(s_1,0)$ and $(s_2,0)$. 
In this case each $s_i,\ i=1,2,3$ is a quasi-modular form of weight $2$
and differential order $1$. They generate the algebra $\mf {}{}(\Gamma(2))$ freely.
 The corresponding differential structure is defined by the Halphen vector field.

\subsection{Algebraic closure of the field of full quasi-modular forms}
In this section we assume that $\k$ is algebraically closed.
Let $f\in \mf nm(\Gamma)$ be a quasi-modular form for a congruence group $\Gamma$. Recall that for $\Gamma=\SL 2\Z$  we write $T=T_\Gamma$. 
Using $f$ we define full quasi-modular forms $f_i,\ i=1,2,\ldots, a:=\Gamma\backslash \SL 2\Z$ of weight $mi$ and differential order $ni$ in the following way:  
$$
f_i: T\to \k,
$$
$$ f_i(E,\omega):=\sum_{*_1,*_2,*_2,\ldots, *_i}f(E,\omega,*_1)f(E,\omega,*_2)\cdots f(E,\omega,*_i),
$$
where $*_1,*_2,*_2,\ldots, *_i$ runs through $i$-tuple of torsion structures attached to $\Gamma$. From Proposition \ref{domingodedanca} it follows that
$$
f_i\in\k[t_1,t_2,t_3],\ f_i \hbox{ homogeneous, }\deg(f_i)=mi, \ \deg_{t_1}(f_i)\leq ni
$$
and $f$ is a root of the polynomial
$$
X^a-f_1X^{a-1}+f_2X^{a-2}-\cdots+(-1)^{a-1}f_{a-1}X+(-1)^af_a.
$$
\begin{exe}\rm
The algebra $\mf {}{}(\Gamma(2))$ is freely generated by three quasi-modular form $s_1,s_2,s_3$ of weight $2$ and differential order $1$. Show that the above polynomial
for each $s_i$ is  
$$
((X-t_1)^3-\frac{1}{4}t_2(X-t_1)-\frac{1}{4}t_3)^2.
$$
\end{exe}

\subsection{Eisenstein modular forms}
\label{alleisen}
Let $E$ be an elliptic curve over $\k$ and $\omega\in H_\dR^1(E)\backslash F^1$. We 
take the Weierstrass coordinates of the pair $(E,\omega)$ as we have described it in 
Proposition \ref{manga2010}. Let also $t$ be a coordinate system around the point $O$, 
for instance take $t=\frac{x}{y}$. We have $\frac{dx}{y}=Pdt$ for some regular function in a 
neighborhood of $O$. 
Let us write the formal series of $P$ at $O$ and then write it as a derivation of some other 
formal power series $Q(t)$. Therefore,
$$
\frac{dx}{y}=dz,\ \ z:=Q(t)\in \ring\{t\},\ 
$$
$$
Q(0)=0, \ \ \frac{\partial Q}{\partial t}(0)=-2.
$$
We call $z$  the analytic coordinate system on $E$.  Note that the first coefficients in $t^3y, t^2x$ are invertible and so 
$Q(t)$ has coefficients in $\ring$. 
\begin{prop}
\label{paulana}
 We have 
$$
x=\frac{1}{z^2}+\sum_{k=1}^{\infty}g_{2k+2}z^{2k},\ 
$$
and
$$
y=\frac{\partial x}{\partial z}=\frac{-2}{z^3}+\sum_{k=1}^{\infty}2k\cdot g_{2k+2}z^{2k},\ g_{2k+2}\in\ring.
$$
\end{prop}
\begin{proof}
We have $z=Q(t)=-2t+O(t^ 2)$ and write $t$ in terms of $z$, that is, $t=\tilde Q(z)=
\frac{-1}{2}z+O(z^ 2)$. Note that $Q\in \ring\{t\}$ and so $\tilde Q\in \ring\{z\}$. We write $x$ in 
terms of $z$ and we have: $x=\sum_{k=-2}^{\infty}g_{k}z^{k}$ for some $g_k\in\ring$.
The elliptic curve $E$ is invariant under the involution $(x,y)\mapsto (x,-y)$. The coordinates $t$ and $z$ are mapped to 
$-t$ and $-z$, respectively, and $x$ is invariant. This implies that $g_{k}=0$ for all odd integers $k$. Calculating $g_0,g_2$ we see that $g_0=g_2=0$. 
The expansion of $y$ follows from the equality $dx=ydz$.
\end{proof}
\begin{prop}
 The mapping $(E,\omega)\to g_{2k+2}$ is a full modular form of weight $2k+2$. 
\end{prop}
We denote this modular form with $G_{2k+2}$ and we call it the Eisenstein modular form of 
weight  $2k+2$.
\begin{proof}
The third property in the definition of a quasi-modular form follows from the fact that in the process of defining $G_{2k+2}$, all the coefficients are in $\ring$.
For $k\in G_a$,  the Weierstrass coordinates system  of $(E,\omega)\bullet k$ is $(x+k,y)$ and so $\frac{dx}{y}$ does not change. This implies that $g_{2k+2}$'s are invariant under the action of $G_a$.
 For $k\in G_m$,  the Weierstrass coordinates system  of $(E,\omega)\bullet k$ is $(\tilde x,\tilde y)=(k^{-2}x,k^{-3}y)$.
In this coordinates system $\tilde t=kt $ and $\tilde z=kz$ which give us the desired functional property of $g_{2k+2}$'s with respect to the action of $G_m$.  
\end{proof}

\begin{exe}\rm
Show the last piece of the proof of Proposition \ref{paulana}, that is, $g_0=g_2=0$. Calculate $G_4$ and $G_6$.   
\end{exe}

\section{Quasi-modular forms over $\C$}
\subsection{Introduction}
\label{int06}
The name quasi-modular form seems to appear for the first time in the work \cite{kaza95} of M. Kaneko and D. Zagier. In this article they give a direct proof for a formula stated by 
R. Dijkgraaf in \cite{dij95} which  deals with counting ramified covering of elliptic curves, see \S\ref{ellipticramified}. In fact if we want to extend the algebra of full modular forms to an algebra which is 
closed under the canonical derivation then we naturally suspect the existence of the Eisenstein series $E_2$. 
The full functional equation can be derived by consecutive derivations of the 
functional equation of a modular form. This is the way in which the author of the present text rediscovered all these, see \cite{ho06-2}, and for this reason called them differential modular 
forms.

The classical modular or quasi-modular forms are holomorphic functions on the Poincar\'e upper half plane 
which satisfy a functional property with respect to the action of a congruence subgroup of $\SL 2\Z$ on $\uhp$ and have some growth condition at infinity. 
The aim of of this section is to show that what we we have developed so far in the context of algebraic geometry is essentially the 
same as its complex counterpart. The bridge between two notions is the period map which is constructed by elliptic integrals.
\subsection{Quasi-modular forms}
\label{dmf}
In this section we recall the definition of quasi/differential modular forms. 
For more details see \cite{maro05,ho06-2}. 
We use the notations
$A=\mat {a_A}{b_A}{c_A}{d_A}\in\SL 2\R$ and
$$
I=\mat{1}{0}{0}{1},\ T=\mat 1101,\ Q=\mat 0{-1}10.
$$
When there is no confusion we will simply write  $A=\mat abcd$.
We denote by $\uhp$ the Poincar\'e upper half plane and 
$$\ja (A,z):=c_Az+d_A.$$ For $A\in \GL(2,\R)$ and $m\in\Z$ we use the
slash operator
$$
f|_mA=(\det A)^{m-1}\ja(A,z)^{-m}f(Az).
$$
Let $\Gamma$ be a subgroup of $\SL 2\Z$.  For instance, take a congruence group 
of level $N$. This is by definition any subgroup 
$\Gamma\subset \SL 2\Z$ which  contains:
$$
\Gamma(N):=\{A\in \SL 2\Z\mid A\equiv \mat{1}{0}{0}{1}\ (\text{ mod } N)\}.
$$
It follows that $\Gamma$ has finite index in $\SL 2\Z$. Our main examples are $\Gamma(N)$ itself and
$$
\Gamma_0(N):=\{A\in \SL 2\Z\mid A\equiv \mat{*}{*}{0}{*}\ (\text{ mod } N)\},
$$
$$
\Gamma_1(N):=\{A\in \SL 2\Z\mid A\equiv \mat{1}{*}{0}{1}\ (\text{ mod } N)\}.
$$
We define the notion of an $\mf nm(\Gamma)$-function, a quasi-modular form of 
weight $m$ and differential order $n$ for $\Gamma$. For simplicity we write 
$\mf{n}{m}(\Gamma)=\mf{n}{m}$.  For $n=0$
an $\mf 0m$-function is a classical modular form of weight
$m$ on $\uhp$ (see bellow).
 A holomorphic function $f$ on $\uhp$ is called $\mf nm$
if the following two conditions are satisfied:
\begin{enumerate}
\item
There are holomorphic functions $f_i,\ i=0,1,\ldots,n$ on $\uhp$ such that 
\begin{equation}
\label{4feb05}
f|_m A=\sum_{i=0}^{n}\bn{n}{i} c_A^i\ja(A,z)^{-i}f_i, \ \forall A\in \Gamma.
\end{equation}
\item
$f_i\mid_mA, i=0,1,2,\ldots,n$ have finite growths when $\Im(z)$ tends to $+\infty$ for all $A\in\SL 2\Z$.
\end{enumerate}
We will also denote  by $\mf nm$ the set of $\mf nm$-functions and
we set
$$
\mf{}{}:=\sum_{m\in\Z, n\in\N_0}\mf nm
$$
For an $f\in\mf{n}{m}$ we have $f|_mI=f_0$ and so $f_0=f$. 
Note that for an
$\mf nm$-function $f$ the associated functions $f_i$ are unique.
If $f$ is $\mf{n}{m}$-function with the associated functions $f_i$ then $f_i$ is an
$\mf{n-i}{m-2i}$-function with the associated functions $f_{ij}:=f_{i+j}$. The set $\mf{}{}$ is a bigraded differential $\C$-algebra:
$$
\diff{}: \mf nm \to \mf {n+1}{m+2}
$$
If $n\leq n'$ then $\mf nm\subset \mf {n'}{m}$ and
$$
\mf nm\mf {n'}{m'}\subset \mf {n+n'}{m+m'}, \ \mf nm+\mf {n'}{m}=\mf
{n'}m
$$
It is useful to define
\begin{equation}
\label{26feb05} f||_mA:=(\det A)^{m-n-1}\sum_{i=0}^{n}\bn{n}{i}
c_{A^{-1}}^i\ja(A,z)^{i-m}f_{i}(Az),\ A\in \GL (2, \R),\ f\in\mf nm.
\end{equation}
The
equality (\ref{4feb05}) is written in the form
\begin{equation}
\label{13feb05} f=f||_{m}A, \forall A\in\Gamma
\end{equation}
One can prove that
$$
f||_m A=f||_m(BA),\ \forall A\in \GL(2,\R),\ B\in\Gamma,\ f\in\mf{n}{m}.
$$
Using this we can prove that the growth at infinity condition on $f$ is a finite number of conditions for $f||_m\alpha, \ \alpha\in\Gamma\backslash \SL 2\Z$.
The relation of $||_m$ with $\diff{}$ is given by:
\begin{equation}
\label{khordim} \diff{(f||_mA)}= \diff{f}||_{m+2}A,\ \forall  A\in\GL(2,\R).
\end{equation}
Let $A\in\SL 2\Z$. If $f\in\mf{n}{m}(\Gamma)$ with the associated functions $f_i$ then $f||_mA\in\mf{n}{m}(A^{-1}\Gamma A)$ with the associated  functions $f_i||_mA\in \mf{n-i}{m-2i}(A^{-1}\Gamma A) $.

\subsection{$q$-expansion}
Let us  assume that there is $h\in \N$ such that
$$
T_{h}:=\mat{1}{h}{0}{1}\in \Gamma
$$
Take $h>0$ the smallest one. Recall that $\Gamma$ is a normal subgroup of $\SL 2\Z$. 
For an $f\in\mf{n}{m}(\Gamma)$ and $A\in\SL 2\Z$ with $[A]=\alpha\in \Gamma\backslash \SL 2\Z$ we have $(f||_mA)|_mT_{h}=f$ and so 
we can write the Fourier expansion of $f||_mA$ at $\alpha$
$$
f||_mA=\sum_{n=0}^{+\infty}a_n q_{h}^{n},\ a_n\in \C,\
\ q_{h}:=e^{2\pi i h z}.
$$
We have used the growth condition on $f$ to see that the above function in $q_{h}$ is holomorphic at $0$.

\subsection{Period domain}
Quasi-modular forms are best viewed  as holomorphic functions on
the period domain
\begin{equation}
\label{perdomain}
\pedo:= \left \{
\mat {x_1}{x_2}{x_3}{x_4}\in\C\mid, x_i\in\C,\ x_1x_4-x_2x_3=1,\  \Im(x_1\ovl{x_3})>0 \right \}.
\end{equation}
We let the group $\SL 2\Z$ (resp. $G$ in (\ref{2nov10}) with $\k=\C$) act from the left (resp. right) on $\pedo$ by usual multiplication of matrices.
The Poincar\'e upper half plane $\uhp$ is embedded in $\pedo$ in the following
way:
$$
z\rightarrow \tilde z=\mat{z}{-1}{1}{0}.
$$
We denote by $\tilde\uhp$ the image of $\uhp$ under this map. Note that any element of $\pedo$ is equivalent to an element of 
$\tilde \uhp$ under the action of $G$ because:
\begin{equation}
\label{neeb}
\mat{x_1}{x_2}{x_3}{x_4}=\mat{\frac{x_1}{x_3}}{-1}{1}{0}
\mat{x_3}{x_4}{0}{\frac{\det(x)}{x_3}}.
\end{equation}
The map 
$$
J:\GL (2,\R)\times \uhp\to G,\ J(A,z)=\mat{\ja(A,z)}{-c}{0}{\ja(A,z)^{-1}\det(A)}
$$
is an automorphy factor, that is, it satisfies the functional equation:
$$
J(AB,z)=J(A,Bz)J(B,z), \ A,B\in \GL (2,\R),\ z\in\uhp.
$$
This follows from the equality
$$
A
\mat{z}{-1}{1}{0}=\mat{Az}{-1}{1}{0}J(A,z),\
A\in\GL (2,\R),\ z\in\uhp.
$$
\begin{prop}
\label{27oct2010}
Quasi-modular forms  $f\in\mf{n}{m}$ are in a one to one correspondence with holomorphic functions 
$F=\phi(f):\pedo\to \C$ with the following properties:
\begin{enumerate}
\item
The function $F$ is $\Gamma$-invariant.
\item
There are holomorphic functions $F_i:\pedo\to \C,\ i=0,1,\ldots,n$ such that
\begin{equation}
\label{7feb} F(x\cdot g)=k^{-m}\sum _{i=0}^n \bn ni
{k'}^{i}k^{i}F _i(x), \ \forall x\in\pedo,\ g\in G,
\end{equation}
\item
For all $\alpha\in\SL 2\Z$ the restriction of $F_i$ to $\tilde\uhp_\alpha$ has finite growth at infinity,
 where 
$\tilde\uhp_\alpha$ is be the image of $\tilde\uhp$ under the action of $\alpha$ from the left on $\pedo$.
\end{enumerate}
\end{prop}

In fact we have $F_i=\phi(f_i)$.
It is a mere calculation to see that the vector field
\begin{equation}
\label{2.09.08}
X:=x_2\frac{\partial}{\partial x_1}+ x_4\frac{\partial}{\partial x_3}
\end{equation}
is invariant under the action of $\SL 2\Z$ and hence it induces a vector 
field in the quotient $\Gamma\backslash \pedo$. 
Now, viewing quasi-modular forms as functions on $\Gamma\backslash \pedo$, the 
differential operator is given by the vector field $X$.

\section{Elliptic integrals}
\label{ellipticsection}
\subsection{Introduction}
The study of multiple integrals requires an independent study of their domain of integration. For the elliptic integral (\ref{odeio}) we have apparently five 
different domain of integration: $[-\infty,a_1], [a_1,a_2], [a_2,a_3],\ [a_3,+\infty]$, where $a_1,a_2,a_3$ are three consecutive roots of $P$. However, the complexification and then
the geometrization process of such integrals, showed that if we fix the integrand $\omega$ then all such integrals can be written in terms of just two of them. All such integrals up to some
constants, which can be calculated effectively, can be written in the form $\int_{\delta}\omega$, where $\delta$ is an element in the homotopy group $G:=\pi_1(E,b)$ of the elliptic curve 
$E: y^2=P(x)$ with a base point $b\in E$. 
The integral is zero over the group $[G,G]$ generated by the commutators of $G$, that is, elements of the form $\delta_1\delta_2\delta_1^ {-1}\delta_2^{-1}$, and so we can choose 
$\delta$ from the first homology group:
$$
\delta\in H_1(E,\Z):=G/[G,G].
$$ 
The homology group $H_1(E,\Z)$ is a $\Z$-module of rank two and in this way we get our affirmation.
This simple observation can be considered as the beginning of singular homology. H. Poincar\'e  in his book named {\it Analysis Situs}  founded the Algebraic Topology and if we look more 
precisely 
for his motivation, we find his articles on multiple integrals. E. Picard was a first who wrote the treatise {\it Th\'eorie des fonctions  alg\'ebriques de deux variables ind\'ependantes} 
on  double integrals.  For him the  study of the integration domain
was justified by the study of integrals, but after him and in particular for S. Lefschetz, it was done as an independent subject. The classical theorems of Lefschetz  on the 
topology of algebraic varieties are an evidence to 
this fact. One of the beautiful topological theories which arose from the study of elliptic and multiple integrals is the so called Picard-Lefschetz theory. A brief description of this is done in 
\S\ref{monodromysection}. It can be considered as the complex counterpart of Morse theory, however, it is older than it. In families of elliptic curves that we consider the topology
is fixed, all elliptic curves are tori. However, some interesting phenomena happens when we turn around degenerated elliptic curves. This is described by Picard-Lefschetz theory. It says that
when we are dealing with families of elliptic curves, $\SL 2\Z$ and its subgroups are there, even we do not mention  them explicitly.

I have avoided to use Weierstrass uniformization theorem for elliptic curves. Instead, I have tried to use Hodge theoretic arguments in order to derive many statements which follow from it.
The main reason for this is the possibility of the generalizations for other type of  varieties. For an introduction to Hodge theory (Hodge structures, period maps, Torelli problem and etc.) 
the reader is referred to the books of C. Voisin (2003).  
In \S \ref{inversesection} we describe how Eisenstein series appear in the inverse of the period map. The case of $E_4$ and $E_6$ follow from the Weierstrass uniformization theorem, however, 
the case of $E_2$ is not covered by this theorem.  In order to calculate the $q$-expansion of a quasi modular form we need first its differential equation and second, its first coefficients.
In order to calculate such coefficients we need to write down the Taylor series of elliptic integrals in a degenerated elliptic curve. Doing this, we get the formulas of elliptic 
integrals in terms of hypergeometric functions (see \ref{hypergeometricsection}) and immediately after, we give another characterization of the Ramanujan differential equation in terms of 
hypergeometric functions (see \S\ref{perram}). This is in fact the method which G. Halphen used to derive a differential equation 
depending on three parameters of the Gauss 
hypergeometric function. 
We also get the Schwarz map whose image gives us the well-known triangulation of the upper half plane such that each triangle is a fundamental domain for $\SL 2\Z$.
 The bridge between the analytic and algebraic versions of quasi-modular forms is the notion of period map constructed 
from elliptic integrals (see \S\ref{permap}). In this section we also explain this. 


\begin{exe}\rm
Calculate the integrals (\ref{boaa}), where $\delta$ is one of 
$[-\infty,a_1], [a_1,a_2], [a_2,a_3],\ [a_3,+\infty]$ and $a_1,a_2,a_3$ are three consecutive roots of $P$ in terms of two of them.
\end{exe}

\subsection{The monodromy group}
\label{monodromysection}
In the framework of algebraic geometry, the arithmetic group $\SL 2\Z$  is hidden and it appears as the 
monodromy group of the family of elliptic curves $E_t$ over $\C$. In order to calculate such a monodromy group 
we need a machinery called Picard-Lefschetz theory. See for instance \cite{arn} for the local version of such a theory and 
see \cite{lam81} for a global version. 

The elliptic curve  $E_t,\ t\in T$ given by (\ref{khodaya}) as a topological space is a torus and hence $H_1(E_t,\Z)$ is a 
free rank two $\Z$-module. Smooth variations of $t$, gives us the  monodromy representation 
$$
\pi_1(T,b)\to {\rm Iso}(H_1(E_b,\Z)),
$$
where $b$ is a fixed point in $T$. We would like to calculate the image of the monodromy representation in a 
fixed basis of   $\delta_1,\delta_2$ of $H_1(E_b,\Z)$.
A classical way for choosing such a basis is given by
the Picard-Lefschetz theory. 
Fix the parameters $t_1$ and $t_2\not=0$ and let $t_3$ varies.
  Exactly for two values 
$\tilde t_3, \ \check t_3=\pm \sqrt{\frac{t_2^3}{27}}$ of $t_3$, the curve $E_t$ is singular. 
In $E_b-\{\infty\}$ we can take two cycles $\delta_1$ and $\delta_2$ such that $\langle \delta_1,\delta_2\rangle=1$
and $\delta_1$ (resp. $\delta_2$) vanishes along a straight line connecting $b_3$ to $\tilde t_3$ (resp. 
$\check t_3$). The corresponding anti-clockwise monodromy around the critical value 
$\tilde t_3$ (resp $\check t_3$)  can be computed using the Picard-Lefschetz formula:
$$
\delta_1\mapsto \delta_1,\ \delta_2\mapsto \delta_2+\delta_1 \ 
(\hbox{ resp. } \delta_1\mapsto \delta_1-\delta_2,\ \delta_2\mapsto \delta_2).
$$ 
The canonical map $\pi_1(\C\backslash \{\tilde t_3,\check t_3\},t)\rightarrow 
\pi_1(T,t)$ induced by inclusion is a surjection and so the image of  $\pi_1(T,t)$ under the monodromy representation is 
$$
\SL 2\Z=\langle A_1,A_2\rangle ,\ \hbox{ where } A_1:=\mat{1}{0}{1}{1},\ A_2:=\mat{1}{-1}{0}{1}.
$$
Let us explain the above topological picture by the 
following one parameter family of elliptic curves:
$$
E_\psi: y^2-4x^3+12x-4\psi=0
$$
For $b$ a real number between $2$ and $-2$ the elliptic curve $E_b$ intersects the real plane $\R^2$ in two
connected pieces which one of them is an oval and we can take it as $\delta_2$ with the anti clockwise orientation. 
In this example as $\psi$ moves from $-2$ to $2$, $\delta_2$ is born  from the point 
$(-1,0)$ and ends up in the  $\alpha$-shaped piece which is the intersection of $E_{2}$ with $\R^2$. The cycle 
$\delta_1$ lies in the complex domain and it vanishes on the critical point $(1,0)$ as $\psi$ moves to $2$. 
It intersects each connected component of 
$E_b\cap \R^2$ once and it is oriented in such away that $\langle \delta_1,\delta_2\rangle=1$. 
 \begin{figure}[t]
\begin{center}
\includegraphics[width=0.5\textwidth]{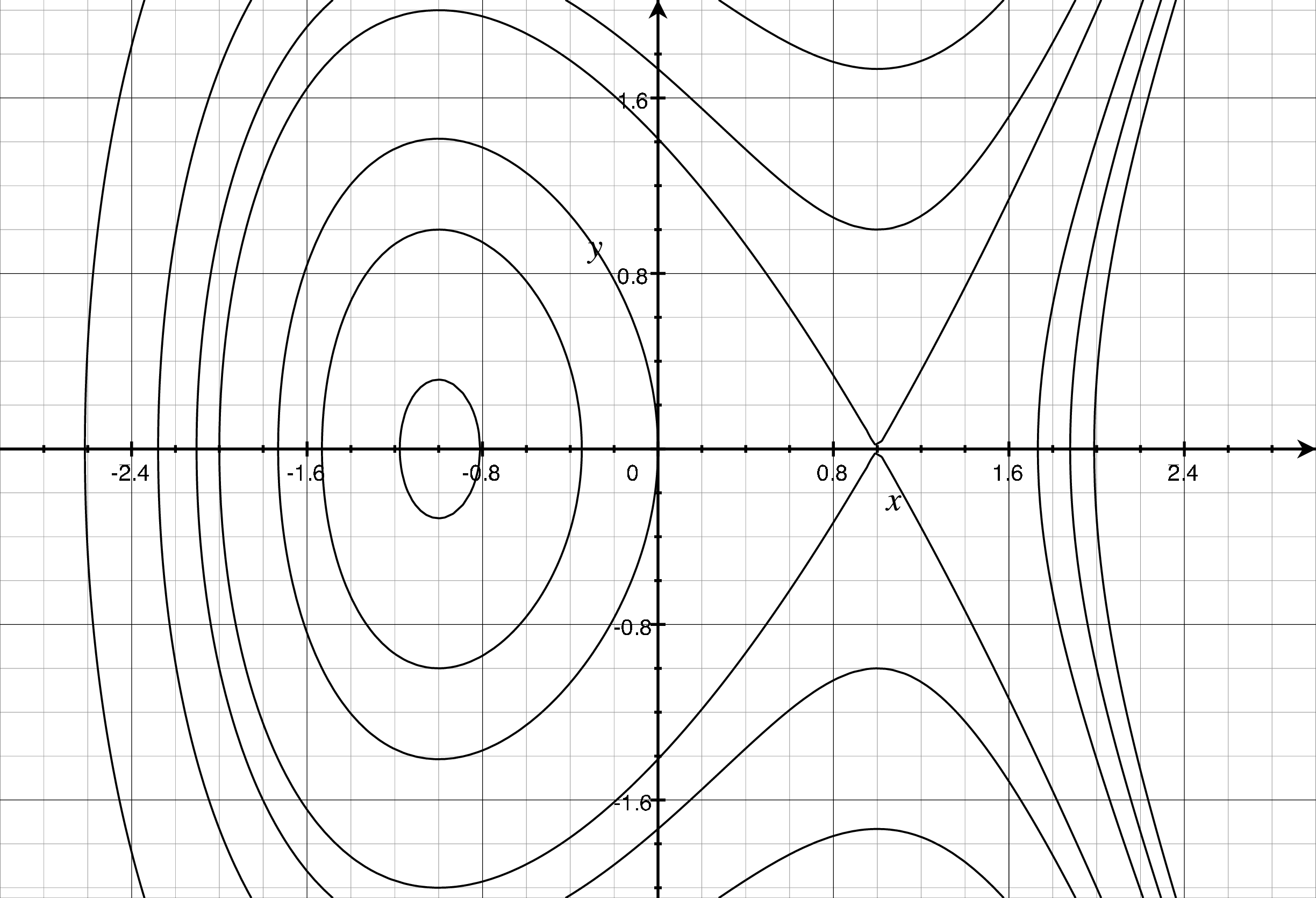}
\caption{Elliptic curves: $y^2-x^3+12x-4\psi,\ \psi=-1.9,-1,0,2,3,5,10$}
\label{eight}
\end{center}
\end{figure}

\begin{exe}\rm
Discuss the details of the content of the present section. If we fix $t_1$ and $t_3$ and let $t_2$ vary then we get three critical curves. Describe the monodromy around each fiber.
\end{exe}

\subsection{Hodge structure of elliptic curves}
\label{hodge}
Let $E$ be an elliptic curves over $\C$. We can regard $E$ as a complex manifold and so we have de Rham cohomologies of
$H_\dR^i(E),\ i=0,1,2$ defined using $C^\infty$ and $\C$-valued differential forms. In the $C^\infty$ context we have also the 
wedge product $H_\dR^1(E)\times H_\dR^1(E)\to H_\dR^2(E)$ bilinear map and the isomorphism 
$$
H_\dR^2(E)\cong \C,\ 
\omega \mapsto \frac{1}{2\pi i}\int_E\omega.
$$ 
The $2\pi i$ factor is there, because in this way the above isomorphism is the complexification of an isomorphism of $\Z$-modules 
$H^2(E,\Z)\cong\Z$. 
The translation of all these in the algebraic context is done in \S \ref{derhamsection} and we 
leave it to the reader the details of comparison of algebraic and $C^\infty$ contexts. For this and the details of what we are going 
to describe the reader is referred to the classical Book of Griffiths and Harris (1978) or to the two volume book of C. Voisin (2002) on Hodge theory. 
 
There is a one dimensional subspace $F^1\subset 
H_\dR^1(E)$ which is spanned by regular differential forms on $E$. We have the complex conjugation in $H_\dR^1(E)$ and it 
turns out that $F^1\cap \overline{ F^1}=\{0\}$ and so
$$
H_\dR^1(E)=F^1\oplus \overline {F^1},
$$
which is called the Hodge decomposition. The bilinear map $\langle \cdot,\cdot \rangle$ 
constructed in \S\ref{intersectionform} turns out to be
$$
\langle \cdot,\cdot \rangle : H_\dR^1(E)\times H_\dR^1(E)\to \C,\ (\omega,\alpha)\mapsto \frac{1}{2\pi i}\int_E \omega\wedge \alpha.
$$
It satisfies the following inequality
\begin{equation}
\label{dominique}
-\sqrt{-1}\langle \omega,\overline{\omega}\rangle>0,\ \omega\in F^1. 
\end{equation}
From all these we want to conclude two well-known facts about elliptic integrals. 
\begin{prop}
\label{aumentodecarga}
Let $E$ be an elliptic curve in the Weierstrass form and let $\delta_1,\delta_2\in H_1(E,\Z)$ with 
$\langle \delta_1,\delta_2\rangle=1$. We have
\begin{enumerate}
 \item 
The integral $\int_{\delta}\frac{dx}{y},\ \delta\in H_1(E,\Z)$ never vanishes and 
\begin{equation}
 \label{parham1}
\Im \left(\frac{\int_{\delta_1}\frac{dx}{y}}{\int_{\delta_2}\frac{dx}{y}}\right )>0.
\end{equation}
\item
We have 
\begin{equation}
\label{parham2}
\int_{\delta_1}\frac{dx}{y}\int_{\delta_2}\frac{xdx}{y}-\int_{\delta_2}\frac{dx}{y}\int_{\delta_1}\frac{xdx}{y}=2\pi i.  
\end{equation}
\end{enumerate}
  \end{prop}
\begin{proof}
The first part follows, for instance, from \cite{si86}, Proposition 5.2.
The second part is known as Legendre relations between elliptic integrals, see for instance \cite{si86}, Exercise 6.4 d. 
We give another proof of all these based on de Rham cohomology arguments presented in this section.

Let $\check \delta_i\in H^1(E_t,\Z), \ i=1,2$ be the Poincar\'e dual of $\delta_i$, that is, $\int_\delta \check \delta_i=
\langle \delta,\delta_i\rangle$ for all $\delta\in H_1(E,\Z)$. The bilinear map $\langle\cdot,\cdot\rangle$ in cohomology is 
dual to the intersection
linear map in cohomology and so $\langle \check\delta_1,\check\delta_2\rangle=1$.  
In the de Rham cohomology $H_\dR^1(E)$ we have
$$
\omega=-(\int_{\delta_2}\omega)\check\delta_1+(\int_{\delta_1}\omega)\check\delta_2, \ \omega=\frac{dx}{y},\ \frac{xdx}{y}.
$$
We use this for $\omega=\frac{dx}{y}$ and we see that the inequality (\ref{dominique}) is equivalent to the 
first part of the proposition. The second 
part follows from
\begin{eqnarray*}
2\pi i &=& 2\pi i\langle \frac{dx}{y},\frac{xdx}{y}\rangle \\
&=& \langle
-(\int_{\delta_2}\frac{dx}{y})\check\delta_1+(\int_{\delta_1}\frac{dx}{y})\check\delta_2, 
-(\int_{\delta_2}\frac{xdx}{y})\check\delta_1+(\int_{\delta_1}\frac{xdx}{y})\check\delta_2\rangle
\\
&=&
(\int_{\delta_1}\frac{dx}{y}\int_{\delta_2}\frac{xdx}{y}-\int_{\delta_2}\frac{dx}{y}\int_{\delta_1}\frac{xdx}{y})
\end{eqnarray*}

\end{proof}

 \begin{exe}\rm
Discuss in more details the material of the present section.
  
 \end{exe}


\subsection{Period map}
\label{permap}
Recall the notations of \S \ref{qmfsection} for the base field 
$\k=\C$. Recall also that that for $\Gamma=\SL 2\Z$ we have
$$
T:=T_\Gamma= \{(t_1,t_2,t_3)\in \C^3\mid 27t_3^2-t_2^3\not =0\}. 
$$
If $\Gamma$ is one of $\Gamma(N),\Gamma_1(N),\Gamma_0(N)$ then we know that 
the projection map $\beta: T_\Gamma\to T$ (neglecting the torsion point structure) is a 
covering of degree $\Gamma\backslash \SL 2\Z$ (see Exercise \ref{14out2010}) and so $T_\Gamma$ 
has a natural structure of a complex manifold. We define
$\Ra_\Gamma$ to be the pull-back of the Ramanujan vector field in $T_\Gamma$.

Let us fix $b\in T_\Gamma$ and  a basis $\delta_1^0,\delta_2^0$ of the $\Z$-module 
$H_1(E_{\beta(b)},\Z)$ with $\langle \delta_1^0,\delta_2^0 \rangle=1$. For any path $\gamma$ 
which connects $b$ to an arbitrary point $t\in T_\Gamma$ we define $\delta_1,\delta_2\in H_1(E_t,\Z)$ to be the
monodromy of $\delta_1^0$ and $\delta_2^0$ along the path $\gamma$.
The period map is defined by 
$$
\per: T_\Gamma\rightarrow \Gamma\backslash \pedo,\ t\mapsto
\left [\frac{1}{\sqrt{2\pi i}}\mat
{\int_{\delta_1} \frac{dx}{y}}
{\int_{\delta_1}\frac{xdx}{y} }
{\int_{\delta_2} \frac{dx}{y}}
{\int_{\delta_2}\frac{xdx}{y} } \right ].
$$
Brackets $[\cdot]$ means the equivalence class in the quotient $\Gamma\backslash \pedo$. 
It is well-defined because of Proposition \ref{aumentodecarga} and the following fact:
different choices of the path $\gamma$ lead to the action of $\Gamma$ from the left on $\pedo$ which is 
already absorbed in the quotient $\Gamma\backslash \pedo$.  Different choices of $b$ and $\delta_1^0,\delta_2^0$ lead to 
the composition of the period map with canonical automorphisms of $\Gamma\backslash\pedo$ (see Exercise \ref{14out2010}, \ref{pmwell}).
\begin{prop}
\label{14oct2010}
We have
\begin{enumerate}
 \item 
The period map is a local biholomorphism; 
\item
It satisfies
\begin{equation}
\label{gavril}
\per(t\bullet g)=\per(t)\cdot g,\ t\in T_\Gamma,\ g\in G;
\end{equation}
\item
The push forward of the vector field $-\Ra_\Gamma$ by the period map $\per$ is the vector field $X$ in (\ref{2.09.08}).
\end{enumerate}
\end{prop}
\begin{proof}
It is enough to prove the Proposition for $\Gamma=\SL 2\Z$ (Exercise \ref{14out2010}, \ref{sl2z}). 
The equality (\ref{gavril}) follows from Proposition \ref{actiont}. The last statement follows from 
Proposition \ref{18.1.06} and (\ref{14.10.10}) as follows:
\begin{eqnarray*}
d\per(\Ra)=\per(t)\cdot A^\tr (\Ra)=\per \mat{0}{0}{-1}{0}=\mat{-x_2}{0}{-x_4}{0}.
\end{eqnarray*}
We have used the notation $\per=\mat{x_1}{x_2}{x_3}{x_4}$. 
Using the equality (\ref{18.10.10}), (\ref{parham2}) we have:
\begin{eqnarray*}
dx_{1}\wedge dx_{3}\wedge dx_{2}& =& A_{11}\wedge A_{12}\wedge(x_{1}A_{3}+x_{2}A_{22})\\
&=&  \frac{1}{\Delta^3}(-\frac{1}{12}d\Delta)\wedge (\frac{3}{2}\alpha)\wedge (x_{1}\Delta dt_1)\\
&=& \frac{3x_{1}}{4\Delta} dt_1\wedge dt_2\wedge dt_3,
\end{eqnarray*}
where $A=[A_{ij}]$ is the Gauss-Manin connection in the basis in Proposition \ref{18.1.06}. 
Using (\ref{parham1}) we conclude that $\per$ is a local biholomorphism.
\end{proof}

\begin{exe}\rm
\label{14out2010}
\begin{enumerate}
\item
For $\Gamma=\Gamma_0(N),\ \Gamma_1(N),\ \Gamma(N)$ show that the cardinality of $\Gamma\backslash \SL 2\Z$ is the number of 
enhanced elliptic curves for $\Gamma$ with $(E,\omega)$ fixed.
\item
\label{pmwell}
Show that the period map $\per$ is well-defined.
\item
For $A\in \Gamma \backslash\SL 2\Z$ we have the well-defined map 
$F_A: \Gamma\backslash \pedo \to \Gamma\backslash \pedo,\ x\mapsto Ax$. A different
choice of $\delta_1^0,\delta^0_2$ in the definition of the period map leads to the composition 
$\per\circ F_A$.
\item 
\label{sl2z}
Proposition \ref{14oct2010} for $\Gamma=\SL 2\Z$ implies the same proposition for arbitrary $\Gamma$.
\end{enumerate}
\end{exe}
\subsection{Inverse of the period map}
\label{inversesection}
In this section we consider the case $\Gamma=\SL 2\Z$. 
Let 
$$
g=(-g_1,g_2,-g_3):\uhp \rightarrow T
$$
be the composition $\uhp\to \SL 2\Z\backslash \pedo\stackrel{\per^{-1}}{\to} T$. We have inserted mines sign because at the end of this
section it will turn out that $g_i$'s are given by (\ref{17nov2010}). 
Here, $\per^{-1}$ is the 
local inverse of the period map, however, since $\uhp$ is simply connected, $g$ is a well-defined one valued 
holomorphic function on $\uhp$. 
For a moment, we assume that the period map is a global
biholomorphism. 
From Proposition \ref{14oct2010} part 2 it follows that $g_i$'s satisfy
\begin{equation}
\label{27.10.10}
(cz+d)^{-2i}g_i(\frac{az+b}{cz+d})=g_i(z),\ i=2,3, 
 \end{equation}
$$
(cz+d)^{-2}g_1(\frac{az+b}{cz+d})=g_1(z)+c(cz+d)^{-1}, \ z\in \uhp, \ \mat{a}{b}{c}{d}\in \SL 2\Z.
$$
 From Proposition \ref{14oct2010} part 3 it follows also that $g$ is a solution of the vector field $\Ra$, that is, 
\begin{equation}
\label{ramanujan} \diff{g_1}=g_1^2- \frac{1}{12}g_2,\
\diff{g_2}=4g_1g_2- 6g_3,\ \diff{g_3}=6g_1g_3- \frac{1}{3}g_2^2
\end{equation}
Since $\mat{1}{1}{0}{1}\in\SL 2\Z$, the functions $g_i$ are invariant under $z\mapsto z+1$, and so,
they can be written in terms of the new variable $q=e^{2\pi i z}$. Later, we will prove that
$g_i$'s have a finite growth at infinity and hence  as functions in $q$ are holomorphic at $q=0$.

\subsection{Hypergeometric functions}
\label{hypergeometricsection}
Let us consider the following one parameter family of elliptic curves
$$
E_\psi: y^2-4x^3+12x-4\psi=0
$$
and the cycles $\delta_1,\delta_2\in H_1(E_\psi,\Z)$ described in \S\ref{monodromysection}: 
for $\psi$ a real number between $-2$ and $2$, $\delta_2$ is the oval vanishing in $(-1,0)$
and $\delta_1\in H_1(E_\psi,\Z)$ vanishes on the nodal point $(1,0)$. Whenever we need to emphasize that $\delta_i, i=1,2$ 
depends on $\psi$ we write $\delta_i=\delta_i(\psi)$.
The cycles $\delta_i,\ i=1,2$ form a basis of $H_1(E_\psi,\Z)$ and it follows from Proposition \ref{18.1.06} and 
the equality (\ref{18.10.10}) that 
the matrix 
$$
Y=\mat{\int_{\delta_1}\frac{dx}{y}}  {\int_{\delta_2}\frac{dx}{y}}{\int_{\delta_1}\frac{xdx}{y}}{\int_{\delta_2}\frac{xdx}{y}}
$$ 
forms a fundamental system of the linear differential equation:
\begin{equation}
\label{No1}
Y'=\frac{1}{\psi^2-4}\mat{\frac{-1}{6}\psi}{\frac{1}{3}}{\frac{-1}{3}}{\frac{1}{6}\psi}Y,
\end{equation}
that is, any solution of (\ref{No1}) is a linear combination of the columns of $Y$.
This example shows a little bit  historical aspects of the Gauss-Manin connection.
From (\ref{No1}) it follows that the elliptic integral $\int_{\delta_2}\frac{dx}{y}$ (resp.   $\int_{\delta_2}\frac{xdx}{y}$) 
satisfies the  differential equation
\begin{equation}
\label{21sep2006}
\frac{5}{36}I+2\psi I'+(\psi^2-4)I''=0 \quad\quad (\text{ resp. }
\frac{ -7}{36}I+2\psi I'+(\psi^2-4)I''=0)
\end{equation}
which is called a Picard-Fuchs equation. We make a linear transformation 
$$
\tau=\frac{\psi+2}{4}
$$
which sends the singularities $\psi=-2,2$ of (\ref{No1}) to $\tau=0,1$. 
We write (\ref{No1}) in the variable $\tau$.  
The integrals $\int_{\delta_2}\frac{dx}{y}$ and $\int_{\delta_2}\frac{xdx}{y}$ are holomorphic around 
$\tau=0$. We write $X:=[\int_{\delta_2}\frac{dx}{y}, \int_{\delta_2}\frac{xdx}{y}]^\tr$ as a formal power series
in $\tau$  $X=\sum_{i=0}^\infty Y_i \tau^i$, substitute it in (\ref{No1}) and obtain a recursive formula for $Y_i$'s. We also obtain 
$Y_0=[a_0,-a_0]^\tr$, where $a_0$ is the value of $\int_{\delta_{2}}\frac{dx}{y}$ at $\psi=-2$. This must be calculated 
separately. The intersection of the elliptic curve $E_\psi, -2<\psi<2$ with the real plane $\R^2$ has two connected component, 
one of them is $\delta_2$ and the other $\tilde\delta_2$ is a closed path in $E_\psi$ which crosses the point at infinity $[0;1;0]$.
It turns out that if we give the clockwise orientation to $\tilde\delta_2$ then it is homotop to $\delta_2$ in $E_\psi$ and
$$
a_0=\left. \int_{\tilde \delta_2} \frac{dx}{y}\right|_{\psi=-2}=
2\int_{2}^\infty \frac{dx}{2(x+1)\sqrt{x-2}}=
\left. \frac{2 {\rm tang}^{-1}(\frac{\sqrt{x-2}}{\sqrt{3}})}{\sqrt{3}}\right|_2^\infty=
\frac{\pi}{\sqrt{3}}.
$$
Note that for $\psi$ a real number  near $-2$, by Stokes formula we have $\int_{\delta_2} \frac{dx}{y}=
\int_{\Delta_2}\frac{dx\wedge dy}{y^2}>0$, where $\Delta_2$ is the region in $\R^2$ bounded by $\delta_2$, and so we already 
knew that $a_0\geq 0$. This explain the fact that why $\delta_2$ is 
homotop to clockwise oriented $\tilde\delta_2$.
The result of all these calculations is:
\begin{equation}
\label{25oct2010}
\int_{\delta_2}\frac{dx}{y}=\frac{\pi}{\sqrt{3}}F(\frac{1}{6},\frac{5}{6},1|\frac{\psi+2}{4}),    
\end{equation}
$$
\int_{\delta_2}\frac{xdx}{y}= -\frac{\pi}{\sqrt{3}} F(\frac{-1}{6},\frac{7}{6},1|\frac{\psi+2}{4}),
$$
where
\begin{equation}
 \label{gauss}
F(a,b,c|z)=\sum_{n=0}^{\infty} \frac{(a)_n(b)_n}{(c)_n n!}z^n, \ c\not\in
\{0,-1,-2,-3,\ldots \},
\end{equation}
is the Gauss hypergeometric function and $(a)_n:=a(a+1)(a+2)\cdots (a+n-1)$.

Let us now calculate the integrals $\int_{\delta_1}\frac{x^idx}{y},\ i=0,1$. We have the isomorphism
 $E_{-\psi}\to E_{\psi},\ (x,y)\mapsto (-x,iy)$ which sends the cycle $\delta_2(-\psi)$ to 
$\delta_1(\psi)$ and $\delta_1(-\psi)$ to $-\delta_2(\psi)$.
This gives us the equalities:
$$
\int_{\delta_1(\psi)}\frac{x^jdx}{y}=(-1)^{j}i \int_{\delta_2(-\psi)}\frac{x^jdx}{y}
$$
Finally, we have calculated all the entries of $Y$:
$$
Y=
\mat
{\frac{\pi i}{\sqrt{3}}F(\frac{1}{6},\frac{5}{6},1|\frac{-\psi+2}{4})}
{\frac{\pi}{\sqrt{3}}F(\frac{1}{6},\frac{5}{6},1|\frac{\psi+2}{4})}
{\frac{\pi i}{\sqrt{3}} F(\frac{-1}{6},\frac{7}{6},1|\frac{-\psi+2}{4})}
{-\frac{\pi}{\sqrt{3}} F(\frac{-1}{6},\frac{7}{6},1|\frac{\psi+2}{4})}
$$
The monodromy around $\tau=0$ leaves $\delta_2$ invariant and takes $\delta_1$ to $\delta_1+\delta_2$. 
From this it follows that for a fixed complex number $a$:
\begin{equation}
 \label{aruq}
\int_{\delta_1}\frac{dx}{y}=\frac{\ln (a \tau)}{2\pi i}(\int_{\delta_2}\frac{dx}{y})+
\frac{1}{2i\sqrt{3}}f(\tau)=
\frac{1}{2i\sqrt{3}}(F(\frac{1}{6},\frac{5}{6},1|\tau)\ln (a\tau)+f(\tau))
\end{equation}
where $f$ is a one valued function in a neighborhood of $\tau=0$. From Exercises \ref{30oct2010}, \ref{creche} 
it follows that $f$ is holomorphic at 
$\tau=0$. We choose $a$ in such a way that the value of $f$ at $\tau=0$ is $0$. This is equivalent to the following formula for 
$a$:
$$
a=\exp(2\pi i (\lim_{\tau\to 0} \int_{\delta_1}\frac{dx}{y}-\frac{\ln \tau}{2\pi i} \int_{\delta_2}\frac{dx}{y})).
$$
According to Exercise \ref{30oct2010}, \ref{creche} we have 
$$
a=\frac{1}{432}.
$$

 We write $f=\sum_{i=1}^\infty f_n\tau^ n$ and substitute (\ref{aruq}) in the Picard-Fuchs 
equation (\ref{25oct2010}) and we obtain the following recursion for $f_n$'s:
$$
f_{n+1}=\frac{(n-\frac{1}{6})(n-\frac{5}{6})}{(n+1)^2}f_n+\frac{(\frac{1}{6})_{n}(\frac{5}{6})_{n}}{ (n!)^2}\frac{2n+1}{(n+1)^2}-
\frac{2}{n+1}\frac{(\frac{1}{6})_{n+1}(\frac{5}{6})_{n+1}}{ ((n+1)!)^2},\ f_0=0. 
$$
We will need the value $f_1=\frac{13}{18}$.


\begin{exe}\rm
\label{30oct2010}
\begin{enumerate}
 \item 
Deduce (\ref{21sep2006}) from (\ref{No1}). 
\item
The integrals $\int_{\delta_2}\frac{dx}{y}$ and $\int_{\delta_2}\frac{xdx}{y}$ are holomorphic at 
$\tau=0$.
\item
Do the details of the calculations which lead to the equalities (\ref{25oct2010}). 
\item
\label{creche}
Prove
$$
\lim_{\tau \to 0}\  \  \frac{\pi i}{\sqrt{3}} F(\frac{1}{6}, \frac{5}{6},1|1-\tau)- \frac{\pi}{\sqrt{3}}F(\frac{1}{6}, \frac{5}{6},1|\tau)\frac{\ln\tau}{2\pi i}=
\frac{\ln (432)}{2\pi i }.
$$
\end{enumerate}

\end{exe}

\subsection{Periods and Ramanujan}
\label{perram}
In this section consider the full modular group $\Gamma=\SL 2\Z$ and the corresponding period map. 
We are interested in the image $L$ of the map $g$ constructed in \S\ref{inversesection}.
This is the locus $L$ of parameters $t\in T$ such that:
\begin{equation}
\label{20oct2010}
\int_{\delta_1}\frac{xdx}{y}=-\sqrt{2\pi i},\ \int_{\delta_2}\frac{xdx}{y}=0,\ \hbox{ for some } \delta_1,\delta_2\in H_1(E_t,\Z) \hbox{ with }
\langle \delta_1,\delta_2\rangle=1.
\end{equation} 
Using Proposition \ref{14oct2010}, part 2 and  and the equality (\ref{neeb}), we know that the locus of such parameters 
is given by:
$$
I=(I_1,I_2,I_3):=(t_1,t_2,t_3)\bullet 
\mat{(\frac{1}{\sqrt{2\pi i}}\int_{\delta_2}\frac{dx}{y})^{-1}}{-\frac{1}{\sqrt{2\pi i}}\int_{\delta_2}\frac{xdx}{y}}{0}{\frac{1}{\sqrt{2\pi i}}\int_{\delta_2}\frac{dx}{y}}=
$$
$$
\left( t_1(2\pi i)^{-1} (\int_{\delta_2}\frac{dx}{y})^{2}-(2\pi i)^{-1}\int_{\delta_2}\frac{xdx}{y}\int_{\delta_2}\frac{dx}{y},\ 
t_2\cdot (2\pi i)^{-2} (\int_{\delta_2}\frac{dx}{y})^{4},\ t_3 \cdot (2\pi i)^ {-3}(\int_{\delta_2}\frac{dx}{y})^{6} \right )
$$
The mentioned locus is one dimensional and the above parametrization is by using three parameters $t_1,t_2,t_3$. 
We may restrict it to a one dimensional subspace $t=(0,12, -4\psi)$ as in \S\ref{hypergeometricsection}, 
use the formulas of elliptic integrals in terms of hypergeometric functions 
(\ref{25oct2010}) and obtain the following parametrization of $L$:
$$
I=\left ( 
-a_1 F(-\frac{1}{6},\frac{7}{6},1\mid \tau)F(\frac{1}{6},\frac{5}{6},1\mid \tau),
a_2 F(\frac{1}{6},\frac{5}{6},1\mid \tau)^{4},
-a_3(1-2\tau) F(\frac{1}{6},\frac{5}{6},1\mid \tau)^{6} \right )
$$
where
$$
(a_1,a_2,a_3)=(\frac{2\pi i}{12}, 12(\frac{2\pi i}{12})^ 2, 8(\frac{2\pi i}{12})^ 3).
$$
From $\nabla_{\Ra}\frac{xdx}{y}=0$ and (\ref{20oct2010}) it follows that $L$ is tangent to the vector field $\Ra$. In other words,
it is a leaf of the foliation induced by $\Ra$. Since the period map sends $\Ra$ to $-X$, and the
canonical map $\uhp\to \pedo$ sends $\frac{\partial }{\partial z}$ to $X$, we conclude that $I_i$'s can be written in 
terms of the new variable
$$
z=\frac{\int_{\delta_1}\frac{dx}{y}}{\int_{\delta_2}\frac{dx}{y}}=
i \frac{F(\frac{1}{6},\frac{5}{6},1|1-\tau)}
{F(\frac{1}{6},\frac{5}{6},1|\tau)}
$$
that is
$$
(I_1,I_2,I_3)=(-g_1(z), g_2(z),-g_3(z)),
$$
where $(g_1,g_2,g_3):\uhp\to \C^3$ is given in \S\ref{inversesection}. Let us define 
$$
E_{2i}(z)=a_i^{-1} g_i(z),\ i=1,2,3.
$$
We get the equalities:
\begin{equation}
\label{vuakala}
 F(-\frac{1}{6},\frac{7}{6},1\mid \tau)F(\frac{1}{6},\frac{5}{6},1\mid \tau)=
E_2(i \frac{F(\frac{1}{6},\frac{5}{6},1|1-\tau)}
{F(\frac{1}{6},\frac{5}{6},1|\tau)}),
\end{equation}
$$
F(\frac{1}{6},\frac{5}{6},1\mid \tau)^{4}
=E_4(i \frac{F(\frac{1}{6},\frac{5}{6},1|1-\tau)}
{F(\frac{1}{6},\frac{5}{6},1|\tau)}),
$$
$$
 (1-2\tau) F(\frac{1}{6},\frac{5}{6},1\mid \tau)^{6}
=E_6(i \frac{F(\frac{1}{6},\frac{5}{6},1|1-\tau)}
{F(\frac{1}{6},\frac{5}{6},1|\tau)}).
$$

\subsection{Torelli problem}
\label{torelli}
In Hodge theory the global injectivity of the period map is known as global Torelli problem. As the reader may have noticed we
need only the local injectivity of the period map in order to extract quasi-modular forms from the 
inverse of the period map.
 
The period map $\per$ in \S\ref{permap} is a global biholomorphism  if and only if 
the induced map $p: T/G\rightarrow \SL 2\Z\backslash \pedo/G\cong \SL2\Z\backslash \uhp$ is a biholomorphism. 
The last 
statement follows from Weierstrass uniformization theorem. We give another proof based on a $q$-expansion argument. 
The quotient $\SL2\Z\backslash \uhp$ has a canonical structure of a Riemann surface such that  the map  $p$ is a 
local biholomorphism. Let $U$ be a subset of  $\SL 2\Z\backslash \uhp$  containing all $z$ with  $\Im(z)>1$.
The map 
$$
U\to D(0, e^{-2\pi}),\ z\mapsto q=e^{2\pi i z},
$$
where $D(0,r)$ is a disk in $\C$ with center $0$ and radius $r$, is a coordinate system around each point of 
$U$. Using this map $\bar S:=\SL 2\Z\backslash \uhp\cup \{ \infty\}$ becomes a compact Riemann surface, where the value of the above 
coordinate at $\infty$ is $q=0$. 
From another side $T/G$ admits also the canonical compactification $\overline{T/G}:=\C^3/G$ which is 
obtained by adding the single point  $p=\{\Delta=0\}/G$ to $T/G$. A coordinate system around $p$ for $\overline{T/G}$ 
is given by $(\C,0)\to \overline{T/G}, \ \tau\mapsto (0,12,-4(4\tau-2))$ 
(recall the one parameter family of elliptic curves in \S\ref{hypergeometricsection}). The map 
$p$ written in these coordinates is
\begin{equation}
 \label{dilma}
\tau\mapsto
q=e^{2\pi i \frac{\int_{\delta_1}\frac{dx}{y}}{\int_{\delta_2}\frac{dx}{y}}} =
\frac{1}{432}\tau e^{\frac{f(\tau)}{F(\frac{1}{6},\frac{5}{6},1|\tau)}}=\frac{1}{432}\tau e^\frac{\frac{13}{18}\tau+\cdots}{1+\frac{5}{36}\tau+\cdots}
\end{equation}
This is an invertible map at $\tau=0$.
This implies that $p$ extends to a local biholomorphism $\overline {T/G}\to \bar S$ without critical points.  
Since both the image and domain of this map  are compact 
Riemann surfaces of genus zero, we conclude that $p$ is a global biholomorphism.

\subsection{$q$-expansion}
The subgroup of $\SL 2\Z$ which leaves the set $\{\mat z{-1}10\mid z\in \uhp\}$ invariant is generated by 
the matrix $\mat 1101$, therefore, the variable $q$ gives us a biholomorphism between the image $\tilde \uhp$ of $\uhp$ in  
$\SL 2\Z\backslash \pedo$ and the punctured disc of radius one. In other words, $q$ is a global coordinate system on $\tilde\uhp$.
Therefore, we can write $g_i$'s of \S\ref{perram} in terms of $q$:
$$
g_i:=I_i(p^{-1}(q)).
$$
where $p:(\C,0)\to (\C,0)$ is the map given by (\ref{dilma}). It follows that $g_i$ as a function in $q$ is one valued and 
holomorphic in  the disc of radius one and center $0$. If we write $g_i$ as a formal 
Laurent series in $q$, and substitute in (\ref{vuakala}), then  we get
a recursion for the coefficients of $g_i$'s. There is a better way to calculate such formal power series. We calculate 
the first two coefficients of $g_1$ as above:
$$
E_1=1-24.q+\cdots
$$
The $g_i=a_iE_i$'s as formal power series satisfy the Ramanujan differential equation (\ref{raman}) and so according to discussion in \S\ref{ramode}, we can calculate all the coefficients of $E_i$ knowing the initial values $1$ and $-24$ as above and the 
recursion given by (\ref{raman}).   

\subsection{Schwarz map}
The multivalued function 
$$
p: \C\to \uhp,\ 
\tau \mapsto \frac{\int_{\delta_1}\frac{dx}{y}}{\int_{\delta_2}\frac{dx}{y}}=
i \frac{F(\frac{1}{6},\frac{5}{6},1|1-\tau)}
{F(\frac{1}{6},\frac{5}{6},1|\tau)}
$$
is called the Schwarz map. We summarize its global behavior in the following proposition:
\begin{prop}
\label{1nov10}
 Let 
$$
U:=\{z\in \C \mid \Re(z)<\frac{1}{2}\}\backslash \{z\in\R \mid z\leq 0\}.
$$
and consider the branch of the Schwarz map in $U$ which has pure imaginary values in $0<\tau<\frac{1}{2}$. Its image is the interior 
of the classical 
fundamental domain of the action of $\SL 2\Z$ in $\uhp$ depicted in Picture (\ref{fundo}). 
Its analytic continuation result in the triangulation of
$\uhp$ as in Picture (\ref{fundo}).  
\end{prop}
Basic ingredients of the proof are the global injectivity of the period map discussed in \S\ref{torelli} and the following 
exercise:
\begin{figure}[t]
\begin{center}
\includegraphics[width=0.5\textwidth]{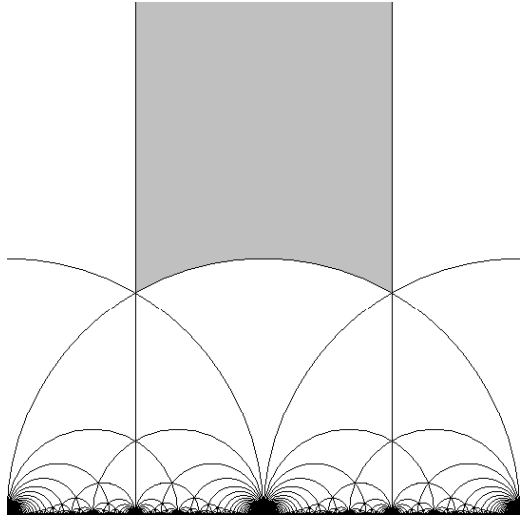}
\caption{Fundamental domain}
\label{fundo}
\end{center}
\end{figure}
\begin{exe}
Let $p$ be the branch of the Schwarz map described in Proposition \ref{1nov10}. Prove the following:
\begin{enumerate}
 \item 
$$
\lim_{\tau\in\R,\ \tau\to 0^+} p(\tau)=+\infty.
$$
\item
$$
|p(\frac{1}{2}+ix)|=1,\ x\in \R.
$$
\item
$$\lim_{x\in\R,\ x\to \pm \infty} p(\frac{1}{2}+ix)=\pm \frac{1}{2}+\frac{\sqrt{3}}{2}.$$
\item
The analytic continuation of $p$ from the upper half (resp. lower half) of $\C$ to $\R^{-}$ has the constant real 
part $\frac{1}{2}$ (resp. $-\frac{1}{2}$).

\end{enumerate}

\end{exe}

%


\subsection{Comparison theorem}

Now, we are in a position to prove that the algebraic and analytic notions of quasi-modular forms are equivalent.
\begin{theo}
\label{main}
 The differential graded algebra of quasi-modular forms  in the Poincar\'e upper 
half plane together with the differential operator $\diff{}$ is isomorphic to 
the  graded differential algebra of quasi-modular forms defined in 
\S \ref{qmfsection} together with the differential operator $\Ra_\Gamma$.
\end{theo}
\begin{proof}
According to Proposition \ref{27oct2010}, quasi modular forms can be viewed as functions on $\Gamma\backslash \pedo$.
Now, the period map which is a biholomorphism gives us the desired isomorphism of 
algebras. 

\end{proof}

\appendix
\section{Quasi-modular forms as sections of jet bundles}
\label{jetbundle}
In this appendix we explain the geometric interpretation of quasi-modular forms in terms 
of sections of jet bundles of tensor powers of line bundles on the moduli of elliptic curves.
The main ingredients of this appendix are taken from a private communication with Prof. P. Deligne. 
We leave to the reader the comparison of the material of this appendix with Lee's article \cite{lee09} and in particular
the notion of quasi-modular polynomial.
For simplicity, we work with elliptic curves over $\C$.

Let $\Gamma\subset \SL 2\Z$ be one of  the modular groups $\Gamma_0(N), \Gamma(N),
\Gamma_0(N)$.
As we noticed $M=\Gamma\backslash \uhp$ is the moduli space of elliptic curves over $\C$ with a certain torsion structure. 
Let $G$ be any algebraic group, 
for instance take  $G$ the multiplicative group $(\C^*,\cdot)$ or the group (\ref{2nov10}) for $\k=\C$. We usually identify 
$G$ with a linear subgroup of $\GL(n,\C)$, for some $n\in \N$, and hence assume that $G$ acts on $\C^n$.  
A $G$-automorphy factor on $\uhp$ is a map:
$$
j: \Gamma\times \uhp \to G 
$$
which satisfies
$$
j(AB,z)=j(A,Bz)j(B,z), \ A,B\in \Gamma,\ z\in\uhp.
$$
Any $G$-automorphy factor gives us a $G$-vector bundle on $M$ and vice verse: the quotient 
$\uhp\times \C^n/\sim$, where 
$$
(z,v)\sim (Az, j(A,z)v),\ \forall z\in \uhp,\  A\in\Gamma, \ v\in \C^ n,
$$ 
gives us a $G$-bundle
in $M$. For $G=(\C^ *,\cdot)$  we get line bundles on $M$. Holomorphic functions $f:\uhp \to\C^n$ with the functional equation
$$
f(Az)=j(A,z)f(z),\ A\in\Gamma, \ z\in\uhp
$$
are in one to one correspondence with holomorphic sections of the $G$-bundle associated to the automorphy factor $j$.

For $z\in \uhp$ we have the elliptic curve $\C/\langle 1,z\rangle$ and the one dimensional vector space 
$\C d\tau$, where $\tau$ is the canonical coordinate on $\C$. This can be also identified with ${\rm Lie}(E)^\vee$, the linear dual of the Lie algebra of $E$, that is, the linear  dual of 
the tangent space of $E$ at $0\in E$.
 This gives us a line bundle, say it $\omega$. The corresponding automorphy 
factor is given by 
\begin{equation}
\label{4nov10}
j(A,z)=(cz+d), \ A=\mat abcd\in \Gamma, \ z\in\uhp.
\end{equation}
(Exercise \ref{4.11.10},\ref{jj2}).
 We have also the canonical line bundle $\Omega$ of $M$ (dual of the tangent bundle of $M$) 
which is given by  the automorphy  factor $j(A,z)^2$. It follows that
\begin{equation}
\label{bh}
\Omega=\omega \otimes \omega.
\end{equation}

For a vector bundle $F\to M$ over a complex manifold $M$, the $n$-th jet bundle $J_nF$ of $F$ is defined as follows: the fiber 
$J_nF_{x}$ of 
$J_nF$ at $x\in M$ is defined to be the set of sections of $F$ in a neighborhood of $x$ modulo those which vanish at $x$ of order $n$.
In other words, $J_nF_{x}$ is the set of Taylor series  at $x$ of the sections of $F$ up to order $n$. A section of $J_nF$ in a small open set $U\subset M$ with a coordinate system 
$(z_1,z_2,\ldots,z_m)$ is a sum
$$
\sum_{0\leq k_1,k_2,\ldots, k_m\leq n}f_{k_1,k_2,\ldots,k_m}(z)(w_1-z_1)^{k_1}(w_2-z_2)^{k_2}\cdots (w_m-z_m)^{k_m},
$$
where $f_{k_1,k_2,\ldots,k_m}(z)$ are holomorphic sections of $F$ in $U$ and $w$ is an extra multi variable. 

Let $M:=\Gamma\backslash \uhp$ and $\omega^m:=\omega\otimes \omega\otimes\cdots\otimes$, $m$ times, be as before. A section of $J_n\omega^m$ corresponds to a sum
$$
F(z,w)=\sum_{i=0}^ n\frac{f_i(z)}{i!}(w-z)^i, \ z\in\uhp, \ w\in\C,
$$
where $f_i$'s are holomorphic functions on $\uhp$ such that
$$
F(Az,Aw)=F(z,w)j(A,w)^m+O((Aw-Az)^{n+1}).
$$
We perform $i$-times the derivation $\frac{\partial }{\partial (Aw)}=(cw+d)^2\frac{\partial }{\partial w}$ in both sides of the above equality and then put $w=z$.
We get functional equations for $f_i$'s:
$$
f_i(Az)=((cz+d)^2\frac{\partial }{\partial z})^ {(i)}(f(z)j(A,z)^m).
$$
Here $f=f_0$ and we redefine the derivation of $f_i$'s: $\frac{\partial f_i}{\partial z}:=f_{i+1}$ (the derivation in other terms is the usual one). Neglecting the growth 
condition for $f_i$'s we conclude that $f_i$ is a quasi-modular form of weight $m+2i$ and differential order $i$.

We discuss briefly the growth condition. From $j(\mat{1}{1}{0}{1},z)=1$ it follows that the line bundles $\omega$ and $\Omega$ have canonical extensions $\bar \omega$ and $\bar \Omega$ to 
$\bar M=\Gamma\backslash (\uhp\cup\Q)$. Sections of $\bar \omega^m$ are in one to one correspondence with modular forms of weight $m$. 
We end this section with the following affirmation:
the following map is a bijection
\begin{equation}
\label{omidpass}
\hbox{ global sections of  }  J_n\bar \omega^{m-2n}\ \  \to \mf {n}{m},
$$
$$
\sum_{i=0}^ n\frac{f_i(z)}{i!}(w-z)^i\mapsto f_n.
\end{equation}


\begin{exe}
\label{4.11.10}
\begin{enumerate}
\item 
\label{jj2}
The automorphy factor associated to $\omega$ and $\Omega$ are respectively $j$ and $j^2$, where $j$ is given by (\ref{4nov10}).
\item
The cohomology bundle $H$ on $\Gamma\backslash \uhp$ associates to each point $z\in\Gamma\backslash \uhp$ the two dimensional 
vector space $H^1(\C/\langle 1,z\rangle,\C)$. We have a canonical inclusion $\omega\subset H$ and an isomorphism of bundles
$H/\omega\cong\omega^{-1}$.
\item
Calculate the automorphy factor of $J_1\omega^{-1}$ and conclude that $J_1\omega^{-1}=H$.
\item
Prove that the Gauss-Manin connection induces an isomorphism
$
\omega\stackrel{\sim}{\to }\Omega\otimes H/\omega.
$
\item
Prove the bijection (\ref{omidpass}).
\end{enumerate}
\end{exe}


\section{Examples of quasi-modular forms as generating functions}
\label{8nov2010}
The field $\Q(E_2,E_4,E_6)$ generated by three Eisenstein series
\begin{equation}
E_{2k}=1-\frac{4k}{B_{2k}}\sum_{n=1}^\infty \left (\sum_{d\mid n}d^{2k-1}\right )q^{n}, \ \  k=1,2,3,
\end{equation}
where $B_2=\frac{1}{6},\ B_4=-\frac{1}{30},\ B_6=\frac{1}{42},\ \ldots$ are Bernoulli numbers,
and its algebraic closure contain many interesting generating functions.  We list some of them without proofs. It is 
convenient to use the weights
$$
\deg(E_{2k})=2k,\ k=1,2,3
$$
for the ring $\Q[E_2,E_4,E_6]$. 


\subsection{Ramified elliptic curves}
\label{ellipticramified}
Let $E$ be a complex elliptic curve and let $p_1,\dots,p_{2g-2}$ be 
distinct points of $E$, where $g>1$. We will discuss the case $g=1$ separately. 
The set $X_{g}(d)$ of  equivalence classes of holomorphic maps $\phi: C\to E$ of degree $d$ from compact connected
smooth  complex curves $C$ to $E$, which have  only one double
ramification point over each point $p_i\in E$ and no other ramification
points, is finite.  By the Hurwitz formula the genus of $C$ is equal to $g$. 
Define
\begin{equation}
\label{pb}
 N_{g,d}:=\sum_{[\phi]\in X_{g}(d)} 
\frac{1}{|\,\text{Aut}\,\,(\phi)\,\,|}
\end{equation}
and 
$$
F_g:=\sum_{d=1}^\infty N_{g,d}{q^d}.
$$
After  R.~Dijkgraaf, M.~Douglas, D.~Zagier, M.~Kaneko, see \cite{dij95, kaza95},  we know that
$$
F_g\in\Q[E_2,E_4,E_6],
$$
For instance, 
$$
F_2(q)=\frac{1}{103680}(10E_2^3-6E_2E_4-4E_6),\
$$
$$
 F_3(q)=\frac{1}{35831808}(-6E_2^6+15E_2^4E_4-12E_2^2E_4^3+7E_4^3+4E_2^3E_6-12E_2E_4E_6+4E_6^2).
$$
For $g=1$ we do note have ramification points and for $\phi: C\to E$ as before, ${\text Aut}(\phi)$ consists of translations by elements of 
$\phi^{-1}(0)$ and so $\#{\text Aut}(\phi))=d$. Therefore, $d\cdot N_{d,1}=\sum_{i|d}i$ is the number of group plus Riemann surface morphisms $C\to E$ of degree $d$. In this case we have the contribution of constant maps which is given by $N_{1,0}\log q=-\frac{1}{24}\log q$. Therefore,
$$
q\frac{\partial F_1}{\partial q}=-\frac{1}{24}E_2.
$$ 
\begin{exe}\rm
Calculate $N_{2,2}$ and $N_{3,3}$ from the formulas for $F_2(q)$ and $F_3(q)$ respectively and prove that in fact they
satisfy (\ref{pb}).
 
\end{exe}

\subsection{Elliptic curves over finite field}
\def\F{\mathbb F}
For an elliptic curve $E$ over a finite field $\F_q$, $q=p^n$ and $p$ a prime number, a theorem of Hasse tells us that there is an algebraic integer 
$\alpha\in \bar\Q$  with $|\alpha|=\sqrt{q}$ and such that
$$
\#E(F_q)=q+1-(\alpha+\bar\alpha).
$$ 
The expression
$$
\sigma_k(q)=-\sum_{E/\F_q}
\frac{(\alpha^{k+1}-\bar\alpha^{k+1})/(\alpha-\bar\alpha)}{\#Aut_{\F_q}(E)}
$$
can be considered as average of the quantities $(\alpha^{k+1}-\bar\alpha^{k+1})/(\alpha-\bar\alpha)$ for all elliptic curves over $\F_q$. Here, $Aut_{\F_q} (E)$ is the group of $\F_q$-automorphisms of E. Let 
$$
\frac{1}{1728}(E_4^3-E_6^2)=q\prod_{n=1}^{\infty}
(1-q^n)^{24}=
(q-24q^2+253q^3-3520q^4+4830q^5+\cdots+\tau(n)q^n+\cdots).
$$
It turns out that
$$
\sigma_{10}(p)=\tau(p),\ \forall p\text{ prime.}
$$
See the article of van der Geer \cite{gee08} for more history behind this phenomenon.

\subsection{Monstrous moonshine conjecture}
We write the $q$ expansion of the  $j$-function
$$
j=
1728\frac{E_4^3}{E_4^3-E_6^2}=
$$
$$
q^{-1}+744+
196884 q+ 21493760q^2 + 864299970{q}^3 + \cdots.
$$
In 1978 MacKay noticed that $196884=196883+1$ and $196883$ is the number of dimensions in which the Monster group can be most simply represented. Based on this observation
J.H. Conway and S.P. Norton in 1979 formulated the 
Monstrous moonshine conjecture which relates all the coefficients in the $j$-function to the representation dimensions of the Monster group.
In 1992 R. Borcherds solved this conjecture and got fields medal. See \cite{gan06} for more information on this conjecture.

\subsection{Modularity theorem}
Let $E$ be an elliptic curve defined by the equation $g(x,y)=0$, where $g$ is a polynomial with integer coefficients and with the discriminant $\Delta\not =0$. For instance,  take $a_2$ and $a_3$  integers with $\Delta:=a_2^3-27a_3^2\not=0$ and let $E: y^2=4x^3-a_2x-a_3$.
Let also $p$ be a prime, $N_p$ be the number of solutions of $E$ working modulo $p$ and $a_p(E):=p-N_p$ (we have not counted the point at infinity $[0;1;0]$).
A version of modularity theorem says that there is an element $f=\sum_{n=0}^\infty a_n q^n$ in the algebraic 
closure of $\Q(E_4,E_6)$ such that 
$a_p=a_p(E)$  for all primes $p\not|\Delta$. In fact $f$ is a cusp form of weight $2$ associated to some 
$\Gamma_0(N)$. Here, $N$ is the conductor of $E$. This was originally known as Taniyama-Shimura conjecture and it is solved by 
A. Wiles, R. Taylor, C. Breuil, B. Conrad, F.Diamond. For further information see \cite{dish05}. As an example
consider $E: y^2+y=x^3-x^2$. This has conductor $N=11$. The corresponding modular form is
$$
\eta(q)^2\eta(q^{11})^{2}=q-2q^2-q^3+2q^4+q^5+2q^6-2q^7-2q^9-2q^{10}+q^{11}-2q^{12}+4q^{13}+\cdots,
$$
where $\eta(q)=q^\frac{1}{24}\prod_{n=1}^{\infty}
(1-q^n) $ is the Dedekind eta function.
For further examples see \cite{hoffman}.

\subsection{Rational curves on $K3$ surfaces}
A $K3$ surface by definition is a simply connected complex surface with trivial canonical bundle.  Projective $K3$ surfaces fall into countable many families ${\cal F}_k,\ k\in \N$.  A surface in ${\cal F}_k$
admits a $k$-dimensional linear system $|L|$ of curves  of genus $k$. 
 A curve $C$ in $|L|$ depends on $k$-parameters and so if we put $k$ conditions on that curve, we would get an isolated  curve and so we can count the number of such curves. 
For instance take $k=n+g,\ n,g\in\N$ and assume that $C$ passes through $g$ generic fixed points and it is singular with $n$ nodal singularities (and hence the geometric genus of $C$ is $g$). 
In fact for a generic $K3$ surface the number $N_n(g)$ of such curves turns out to be finite. 
The generating function for the numbers $N_n(0)$, that is the number of rational curves in the linear system $|L|$,  was first discovered by 
Yau-Zaslow (1996),  Beauville (1999) and G\"ottsche (1994):
$$
\sum_{n=0}^\infty N_{n}(0)q^{n}=
\frac{1728q}{E_4^3-E_6^2}=1 + 24q + 324q^2 + 3200q^3 + 25650q^4 + 176256q^5 + 1073720q^6 +\cdots
$$
(by definition $N_0(0)=1$).  For instance,  a smooth quadric $X$ in  $\P^ 3$ is $K3$ and for such a generic $X$ the number of planes tangent to $X$  in three points is $3200$.
  
For arbitrary genus  $g$ we have the following generalization  of Bryan-Leung (1999):
$$
\sum_{n=0}^\infty N_{n}(g)q^{n}=
(\frac{-1}{24}\frac{\partial E_{2}}{\partial q})^{g}\frac{1728q}{E_4^3-E_6^2}.
$$
Some first coefficients for $g=1$ and $g=2$ are given respectively by
$$
1 + 30q + 480q^2 + 5460q^3 + \cdots,\ 1 + 36q + 672q^2 + 8728q^3 + \cdots.
$$



\def\cprime{$'$} \def\cprime{$'$} \def\cprime{$'$}

\end{document}